\documentclass{siamonline171218}

\RequirePackage[normalem]{ulem}
\RequirePackage{color}\definecolor{RED}{rgb}{1,0,0}\definecolor{BLUE}{rgb}{0,0,1}

\usepackage{lipsum}
\usepackage{amsfonts}
\usepackage{graphicx}
\usepackage{epstopdf}
\usepackage{algorithmic}
\ifpdf
  \DeclareGraphicsExtensions{.eps,.pdf,.png,.jpg}
\else
  \DeclareGraphicsExtensions{.eps}
\fi

\numberwithin{theorem}{section}

\newcommand{\TheTitle}{Higher-Order Total Directional Variation: Analysis} 
\newcommand{\TheAuthors}{S. Parisotto, S. Masnou, and C.-B. Sch{\"{o}}nlieb}

\headers{\TheTitle}{\TheAuthors}

\title{
{\TheTitle}
\thanks{
Submitted to the editors DATE.
\funding{
SP acknowledges UK EPSRC grant EP/L016516/1 for the University of Cambridge, Cambridge Centre for Analysis DTC.  SM acknowledges support from the French National Research Agency (ANR) research grant MIRIAM (ANR-14-CE27- 0019) and the European Union Horizon 2020 research and innovation programme under the Marie Sk\l{}odowska-Curie grant agreement No 777826 (NoMADS).
CBS acknowledges support from
the EPSRC grants Nr. EP/M00483X/1, EP/K009745/1,
the EPSRC centre EP/N014588/1,
the Leverhulme Trust project 'Breaking the non-convexity barrier',
the Alan Turing Institute TU/B/000071, 
the CHiPS (Horizon 2020 RISE project grant), 
the Isaac Newton Institute 
and 
the Cantab Capital Institute for the Mathematics of Information. 
 }
}
}

\author{
Simone Parisotto\thanks{CCA, University of Cambridge, Wilberforce Road,
Cambridge CB3 0WA, UK (\email{sp751@cam.ac.uk})}
\and
Simon Masnou\thanks{Univ Lyon, Universit\'e Claude Bernard Lyon 1, CNRS UMR 5208, Institut Camille Jordan, 69622 Villeurbanne, France (\email{masnou@math.univ-lyon1.fr})} 
\and 
Carola-Bibiane Sch\"onlieb\thanks{DAMTP, University of Cambridge, Wilberforce Road,
Cambridge CB3 0WA, UK (\url{cbs31@cam.ac.uk})}
}

\usepackage{amsopn}
\usepackage{stmaryrd}
\DeclareMathOperator{\diag}{diag}

\ifpdf
\hypersetup{
  pdftitle={\TheTitle},
  pdfauthor={\TheAuthors}
}
\fi

\usepackage[english]{babel}
\usepackage[utf8]{inputenc} 

\usepackage{tipa}
\usepackage{amsmath}
\usepackage{amsfonts}
\usepackage{amssymb}
\usepackage{mathtools}
\usepackage{cleveref}

\usepackage[optno]{optional}

\usepackage{epigraph}
\usepackage{pifont}

\usepackage[super]{nth}

\ifx\doctype\doctypepresentation
\else
\usepackage{footnote}
\makesavenoteenv{tabular}
\makesavenoteenv{table}
\fi

\usepackage{scalerel,stackengine}
\stackMath
\newcommand\widecheck[1]{\savestack{\tmpbox}{\stretchto{\scaleto{\scalerel*[\widthof{\ensuremath{#1}}]{\kern-.6pt\bigwedge\kern-.6pt}{\rule[-\textheight/2]{1ex}{\textheight}}}{\textheight}}{0.5ex}}\stackon[1pt]{#1}{\scalebox{-1}{\tmpbox}}}

\renewcommand\widehat[1]{\savestack{\tmpbox}{\stretchto{\scaleto{\scalerel*[\widthof{\ensuremath{#1}}]{\kern-.6pt\bigwedge\kern-.6pt}{\rule[-\textheight/2]{1ex}{\textheight}}}{\textheight}}{0.5ex}}\stackon[1pt]{#1}{\tmpbox}}

\AtBeginDocument{

  \let\diag\relax

  \let\div\relax
  \let\dom\relax
    
  \let\grad\relax

  \let\tr\relax

  \DeclareMathOperator{\diag}{diag}

  \DeclareMathOperator{\div}{div}
  \DeclareMathOperator{\dom}{dom}
  
  \DeclareMathOperator*{\esssup}{ess\,sup}
  \DeclareMathOperator{\grad}{\bm{\nabla}}

  \DeclareMathOperator{\tr}{trace}
}
 \usepackage{calrsfs}
\DeclareMathAlphabet{\pazocal}{OMS}{zplm}{m}{n}

\newcommand{\Mvarcal}{\mathcal{M}}

\newcommand{\Dcal}{\pazocal{D}}
\newcommand{\Ecal}{\pazocal{E}}

\newcommand{\Kcal}{\pazocal{K}}

\newcommand{\Hcal}{\pazocal{H}}
\newcommand{\Lcal}{\pazocal{L}}
\newcommand{\Mcal}{\pazocal{M}}

\newcommand{\Scal}{\pazocal{S}}
\newcommand{\Tcal}{\pazocal{T}}

\newcommand{\Ycal}{\pazocal{Y}}

\newcommand{\Drm}{\mathrm{D}}

\newcommand{\Qrm}{\mathrm{Q}}

\newcommand{\Abold}{\mathbf{A}}

\newcommand{\Ibold}{\mathbf{I}}

\newcommand{\Mbold}{\mathbf{M}}

\newcommand{\Rbold}{\mathbf{R}}

\newcommand{\abold}{\bm{a}}

\newcommand{\ebold}{\bm{e}}
\newcommand{\fbold}{\bm{f}}

\newcommand{\hbold}{\bm{h}}

\newcommand{\ubold}{\bm{u}}
\newcommand{\vbold}{\bm{v}}
\newcommand{\xbold}{\bm{x}}
\newcommand{\ybold}{\bm{y}}
\newcommand{\wbold}{\bm{w}}
\newcommand{\zbold}{\bm{z}}

\usepackage{bm} \newcommand{\alphabold}{{\bm{\alpha}}}
\newcommand{\betabold}{{\bm{\beta}}}
\newcommand{\etabold}{{\bm{\eta}}}

\newcommand{\xibold}{{\bm{\xi}}}
\newcommand{\Phibold}{{\bm{\Phi}}}
\newcommand{\Psibold}{{\bm{\Psi}}}

\newcommand{\Lambdabold}{{\bm{\Lambda}}}

 \ifx\doctype\doctypepresentation
\usepackage{color}
\else
\usepackage[table]{xcolor}
\fi

\definecolor{darkred}{rgb}{0.55,0.0,0.0}
\definecolor{darkgreen}{rgb}{0,0.55,0.0}
\definecolor{darkblue}{rgb}{0,0.0,0.55}
\definecolor{RoyalRed}{rgb}{0.6179,0.0236,0.0894} 
\definecolor{deepblue}{rgb}{0,0,0.5}
\definecolor{deepred}{rgb}{0.6,0,0}
\definecolor{deepgreen}{rgb}{0,0.5,0}
\definecolor{grey}{gray}{0.5}
\definecolor{lightgrey}{gray}{0.8}
\definecolor{lightblue}{rgb}{0.4,0.6,0.8}
\definecolor{lightred}{rgb}{0.8,0.6,0.4}

\colorlet{commentcolour}{green!50!black}
\colorlet{stringcolour}{red!60!black}
\colorlet{keywordcolour}{magenta!90!black}
\colorlet{exceptioncolour}{yellow!50!red}
\colorlet{commandcolour}{blue!60!black}
\colorlet{numpycolour}{blue!60!green}
\colorlet{literatecolour}{magenta!90!black}
\colorlet{promptcolour}{green!50!black}
\colorlet{specmethodcolour}{violet}
\colorlet{indendifiercolour}{green!70!white} 
\usepackage{xparse}
\DeclareDocumentCommand\CCCspace{mo}
 {\IfValueTF{#2}
   {\mathrm{{C}}^{#1}(#2)}
   {\mathrm{{C}}^{#1}}}
 \DeclareDocumentCommand\CCospace{mo}
 {\IfValueTF{#2}
   {\mathrm{{C}}_0^{#1}(#2)}
   {\mathrm{{C}}_0^{#1}}}
 
 \DeclareDocumentCommand\CCCcomp{mo}
 {\IfValueTF{#2}
   {\mathrm{{C}}_c^{#1}(#2)}
   {\mathrm{{C}}_c^{#1}}}

\DeclareDocumentCommand\WW{mo}
 {\IfValueTF{#2}
   {\mathrm{{W}}^{#1}(#2)}
   {\mathrm{{W}}^{#1}}}
 
\DeclareDocumentCommand\WWloc{mo}
 {\IfValueTF{#2}
   {\mathrm{{W}}_{\text{loc}}^{#1}(#2)}
   {\mathrm{{W}}_{\text{loc}}^{#1}}}
 
\DeclareDocumentCommand\WWo{mo}
 {\IfValueTF{#2}
   {\mathrm{{W}}_0^{#1}(#2)}
   {\mathrm{{W}}_0^{#1}}}
 
\DeclareDocumentCommand\HH{mo}
 {\IfValueTF{#2}
   {\mathrm{{H}}^{#1}(#2)}
   {\mathrm{{H}}^{#1}}}
 
\DeclareDocumentCommand\HHo{mo}
 {\IfValueTF{#2}
   {\mathrm{{H}}_0^{#1}(#2)}
   {\mathrm{{H}}_0^{#1}}}

\DeclareDocumentCommand\LL{mo}
 {\IfValueTF{#2}
   {\mathrm{{L}}^{#1}(#2)}
   {\mathrm{{L}}^{#1}}}
 
\DeclareDocumentCommand\LLc{mo}
 {\IfValueTF{#2}
   {\mathrm{{L}}_\mathrm{c}^{#1}(#2)}
   {\mathrm{{L}}_\mathrm{c}^{#1}}}
 
\DeclareDocumentCommand\LLloc{mo}
 {\IfValueTF{#2}
   {\mathrm{{L}}_{\text{loc}}^{#1}(#2)}
   {\mathrm{{L}}_{\text{loc}}^{#1}}} 
 
\DeclareDocumentCommand\BH{mo}
 {\IfValueTF{#2}
   {\mathrm{{BH}}^{#1}(#2)}
   {\mathrm{{BH}}^{#1}}}
 
\DeclareDocumentCommand\BV{mo}
 {\IfValueTF{#2}
   {\mathrm{{BV}}^{#1}(#2)}
   {\mathrm{{BV}}^{#1}}}
 
\DeclareDocumentCommand\BDV{mmo}
 {\IfValueTF{#3}
   {\mathrm{{BDV}}^{(#1,#2)}(#3)}
   {\mathrm{{BDV}}^{(#1,#2)}}}
 
\DeclareDocumentCommand\BDD{mo}
 {\IfValueTF{#2}
   {\mathrm{{BDD}}^{#1}(#2)}
   {\mathrm{{BDD}}^{#1}}}
 
 \DeclareDocumentCommand\BDVM{mmo}
 {\IfValueTF{#3}
   {\mathrm{{BDV}}^{#1}_{#2}(#3)}
   {\mathrm{{BDV}}^{#1}_{#2}}}
 
\DeclareDocumentCommand\BDVb{mmmo}
 {\IfValueTF{#4}
   {\mathrm{{BDV}}_{#3}^{(#1,#2)}(#4)}
   {\mathrm{{BDV}}_{#3}^{(#1,#2)}}}
 
\DeclareDocumentCommand\BDVloc{mmo}
 {\IfValueTF{#3}
   {\mathrm{{BDV}}_{\text{loc}}^{(#1,#2)}(#3)}
   {\mathrm{{BDV}}_{\text{loc}}^{(#1,#2)}}} 
 
\DeclareDocumentCommand\BGDV{mmo}
 {\IfValueTF{#3}
   {\mathrm{{BGDV}}^{(#1,#2)}(#3)}
   {\mathrm{{BGDV}}^{(#1,#2)}}}
 
\DeclareDocumentCommand\BGV{mo}
 {\IfValueTF{#2}
   {\mathrm{{BGV}}^{#1}(#2)}
   {\mathrm{{BGV}}^{#1}}}
 
\DeclareDocumentCommand\BD{mo}
 {\IfValueTF{#2}
   {\mathrm{{BD}}^{#1}(#2)}
   {\mathrm{{BD}}^{#1}}} 
 
\DeclareDocumentCommand\BVloc{mo}
 {\IfValueTF{#2}
   {\mathrm{{BV}}_{\text{loc}}^{#1}(#2)}
   {\mathrm{{BV}}_{\text{loc}}^{#1}}}
\DeclareDocumentCommand\SBV{mo}
 {\IfValueTF{#2}
   {\mathrm{{SBV}}^{#1}(#2)}
   {\mathrm{{SBV}}^{#1}}}
 
\DeclareDocumentCommand\TV{mo}
 {\IfValueTF{#2}
   {\mathrm{{TV}}^{#1}(#2)}
   {\mathrm{{TV}}^{#1}}}
 
  \DeclareDocumentCommand\Sym{mo}
 {\IfValueTF{#2}
   {\mathrm{{Sym}}^{#1}(#2)}
   {\mathrm{{Sym}}^{#1}}}
\DeclareDocumentCommand\TGV{mmo}
 {\IfValueTF{#3}
   {\mathrm{{TGV}}_{#1}^{#2}(#3)}
   {\mathrm{{TGV}}_{#1}^{#2}}}
\DeclareDocumentCommand\TGDV{mmmo}
 {\IfValueTF{#4}
   {\mathrm{{TGDV}}^{#1}_{#2, #3}(#4)}
   {\mathrm{{TGDV}}^{#1}_{#2, #3}}} 
 
 \DeclareDocumentCommand\PTV{mo}
 {\IfValueTF{#2}
   {\mathrm{{PTV}}^{#1}(#2)}
   {\mathrm{{PTV}}^{#1}}}

 \DeclareDocumentCommand\TDV{mmo}
 {\IfValueTF{#3}
   {\mathrm{{TDV}}^{(#1,#2)}({#3})}
   {\mathrm{{TDV}}^{(#1,#2)}}}
 
  \DeclareDocumentCommand\TDD{mo}
 {\IfValueTF{#2}
   {\mathrm{{TDD}}^{#1}({#2})}
   {\mathrm{{TDD}}^{#1}}}
 
  \DeclareDocumentCommand\TDVb{mmmo}
 {\IfValueTF{#4}
   {\mathrm{{TDV}}_{{#3}}^{(#1,#2)}({#4})}
   {\mathrm{{TDV}}_{{#3}}^{(#1,#2)}}}
 
   \DeclareDocumentCommand\TDVM{mmo}
 {\IfValueTF{#3}
   {\mathrm{{TDV}}_{{#2}}^{#1}({#3})}
   {\mathrm{{TDV}}_{{#2}}^{#1}}}
 
   \DeclareDocumentCommand\DVM{mmo}
 {\IfValueTF{#3}
   {\mathrm{{DV}}_{{#2}}^{#1}({#3})}
   {\mathrm{{DV}}_{{#2}}^{#1}}} 
 
  \DeclareDocumentCommand\TDVh{mmo}
 {\IfValueTF{#3}
   {\mathrm{{TDV}}_h^{(#1,#2)}(#3)}
   {\mathrm{{TDV}}_h^{(#1,#2)}}}
 
\DeclareDocumentCommand\DTV{mo}
 {\IfValueTF{#2}
   {\mathrm{{DTV}}^{#1}(#2)}
   {\mathrm{{DTV}}^{#1}}} 
 
 \DeclareDocumentCommand\DTGV{mmo}
 {\IfValueTF{#3}
   {\mathrm{{DTGV}}^{#1}_{#2}(#3)}
   {\mathrm{{DTGV}}^{#1}_{#2}}} 
 
 \DeclareDocumentCommand\EXP{mo}
 {\IfValueTF{#2}
   {\EE_{#2}\left[#1\right]}
   {\EE\left[#1\right]}}

 \usepackage[algo2e]{algorithm2e}

\usepackage[formats]{listings}

\newcommand{\codetitlestyle}[1]{\small\textit{#1}}
\newcommand{\belowtitleskip}{2pt}

\usepackage{listingsutf8}
\DeclareFixedFont{\ttb}{T1}{txtt}{bx}{n}{12} \DeclareFixedFont{\ttm}{T1}{txtt}{m}{n}{12}  

\lstdefinestyle{pythonstyle}{
language=Python,
basicstyle=\ttfamily\small,
frame=tb,                         }

\lstdefinestyle{matlabstyle}{
language=matlab,
basicstyle=\ttfamily\small,
frame=tb,
rulecolor=\color{black!40},
emphstyle=\color{blue},
commentstyle=\color{commentcolour}\slshape,
}

\lstdefinestyle{matlab}{frame=single,                           basicstyle=\scriptsize\ttfamily,             keywordstyle=[1]\color{darkblue}\bfseries,        keywordstyle=[2]\color{darkred},         keywordstyle=[3]\color{darkred}\underbar,  identifierstyle=,                       commentstyle=\usefont{T1}{pcr}{m}{sl}\color{darkred}\small,
        stringstyle=\color{darkred},             showstringspaces=false,                 tabsize=5,                              morekeywords={xlim,ylim,var,alpha,factorial,poissrnd,normpdf,normcdf},
morekeywords=[2]{on, off, interp},
morekeywords=[3]{FindESS, homework_example},
morecomment=[l][\color{darkred}]{...},     numbers=left,                           firstnumber=1,                          numberstyle=\tiny\color{darkred},          stepnumber=5                            }

\lstnewenvironment{python}[1][]
{
\pythonstyle
\lstset{#1}
}
{}

\newcommand\pythoninline[1]{{\pythonstyle\lstinline!#1!}}

\lstnewenvironment{matlab}[1][]{\lstset{style=matlabstyle,title={\codetitlestyle{MATLAB code }},belowcaptionskip=\belowtitleskip}}{}

\newcommand{\RR}{\mathbb{R}}

\newcommand{\NN}{\mathbb{N}}

\newcommand{\EE}{\mathbb{E}}

\newcommand{\diff}{\mathop{}\mathrm{d}}

\newcommand{\at}[2][]{#1|_{#2}}

\newcommand{\weakconv}{\rightharpoonup}
\newcommand{\weakstarconv}{\xrightharpoonup{\ast}}

\newcommand{\abs}[1]{\left\lvert #1 \right\rvert}
\newcommand{\norm}[1]{\left\lVert #1 \right\rVert}

\newcommand{\blank}{\,{\cdot}\,}

   \newcommand{\T}{\mathrm{T}}

\renewcommand{\Im}{\mathrm{Im}} 
\usepackage[optno]{optional}

\usepackage[binary-units=true]{siunitx}

\usepackage{mathtools}

\usepackage[font=scriptsize]{subcaption}
\usepackage{tikz}
\usetikzlibrary{arrows,matrix,positioning,scopes,calc}

\usepackage{graphicx}
\graphicspath{{./../images/}}
\usepackage{tabularx,longtable,booktabs,wrapfig,multirow}

\usepackage{epstopdf}
\AppendGraphicsExtensions{.tif}

\AppendGraphicsExtensions{.tif}

\newsiamremark{example}{Example}
\newsiamremark{remark}{Remark}

\usepackage[breakable,most]{tcolorbox}

\usepackage{empheq}
\tcbset{
  highlight math style={
    enhanced,
    colframe=black,
    colback=white,
    arc=0pt,
    boxrule=1pt,
  }
}

\newenvironment{boxedeq*}
  {\empheq[box=\tcbhighmath]{equation*}}
  {\endempheq}

\newenvironment{boxedalign*}
  {\empheq[box=\tcbhighmath]{align*}}
  {\endempheq}

\crefformat{figure}{#2Figure~#1#3}
\crefformat{algorithm}{#2Algorithm~#1#3}
\crefformat{section}{#2Section~#1#3}
\crefformat{subsection}{#2Section~#1#3}
\crefformat{chapter}{#2Chapter~#1#3}
\crefformat{equation}{#2(#1)#3}
\crefformat{table}{#2Table~#1#3}
\crefformat{theorem}{#2Theorem~#1#3}
\crefformat{lemma}{#2Lemma~#1#3}
\crefformat{remark}{#2Remark~#1#3}
\crefformat{example}{#2Example~#1#3}
\crefformat{proposition}{#2Proposition~#1#3}
\crefformat{property}{#2Propery~#1#3}
\crefformat{definition}{#2Definition~#1#3}
\crefrangeformat{figure}{Figures~(#3#1#4--#5#2#6)}
\crefrangeformat{equation}{(#3#1#4)--(#5#2#6)}

\allowdisplaybreaks

 \usepackage{enumitem}
\usepackage{calc}
\usepackage{mathtools}

\begin{document}

\maketitle

\begin{abstract}
We analyse a new notion of total anisotropic higher-order variation which, differently from the Total Generalized Variation by Bredies et al., quantifies for possibly non-symmetric tensor fields their variations at arbitrary order weighted by possibly inhomogeneous, smooth elliptic anisotropies. 
We prove some properties of this total variation and of the associated spaces of tensors with finite variations. We show the existence of solutions to a related regularity-fidelity optimisation problem. We also prove a decomposition formula which appears to be helpful for the design of numerical schemes, as shown in a companion paper, where several applications to image processing are studied.

\end{abstract}

\begin{keywords}
  Anisotropic total variation, higher-order total variation, variational model
\end{keywords}

\begin{AMS}
  47A52, 49M30, 49N45, 65J22, 94A08
\end{AMS}

\section{Introduction}
Total variation ($\TV{}$) regularisation is one of the most prominent regularisation approaches, successfully applied in a variety of imaging problems. Indeed, since \cite{ROF},  $\TV{}$ played a crucial role for image denoising, image deblurring, inpainting, magnetic resonance image (MRI) reconstruction and many others, see \cite{ChCaCrNoPo10}. 
Extensions of total variation regularisation are TV-type regularisers that feature higher-order derivatives \cite{ChaEsePar2010, ChanMarquinaMulet2000, Papafitsoros2014, SetzerSteidl, WuTai,BreKunPoc2010} -- in particular accommodating for more complex image structures and countering certain $\TV{}$ artefacts such as staircasing -- as well as $\TV{}$ regularisers that encode directional information -- so as to enhance the quality of the smoothing results along preferred directions -- e.g.\ \cite{bayram2012directional, EADTV, BeBuDr06, Dong2009, Steidl2009, GL10, lenzen2013class, STV, EstSoaBre15, Ehrhardt16, directionaltv}. Very general anisotropies have also been studied, as in \cite{AmaBel94}, where it is shown that a fairly general class of metrics, possibly discontinuous, yields a well-defined notion of first-order anisotropic total variation.

In this paper we consider a new class of TV-type regularisers that we have recently introduced in \cite{ParMasSch18applied} and called total directional variation ($\TDVM{}{}$). These regularisers extend the higher-order $\TV{}$ of \cite{BreKunPoc2010} (the so-called total generalized variation, $\TGV{}{}$, see below), and the directional total generalized variation of \cite{directionaltv} (which promotes smoothness along a single, constant direction), to higher-order $\TV{}$ regularisation with spatially-varying directional smoothing. This is done by means of weighting derivatives with $2$-tensors, see below. In \cite{ParMasSch18applied} we propose the $\TDVM{}{}$ regulariser, discuss its discretisation and numerical solution, and demonstrate its performance on a range of imaging applications such as image denoising, wavelet-based zooming, and digital elevation map (DEM) interpolation with applications to atomic force microscopy (AFM) data. In this paper we give a theoretical analysis of the $\TDVM{}{}$ regulariser in the continuum.

Let $\Omega$ be a bounded Lipschitz domain. 
We address the analysis of the higher-order total directional variation defined for every tensor-valued function $\ubold:\Omega\to\Tcal^\ell(\RR^d)$, with $\Tcal^\ell(\RR^d)$ the vector space of $\ell$-tensors in $\RR^d$ and $\ell\in\NN$, as
\begin{equation}
\TDVM{\Qrm,\ell}{\alphabold}[\ubold,\Mcal]:= \sup\left\{ \int_\Omega \ubold \cdot \div_{\Mcal}^\Qrm \Psibold \diff \xbold, \,\Psibold \in\Ycal_{\Mcal,\alphabold}^{\Qrm,\ell} \right\},
\label{eq: TDV intro def}
\end{equation}
where $\Qrm$ is the order of regularisation, $\Mcal$ is a collection of  weighting fields acting on each derivative order, $\alphabold$ is the vector regularisation parameter and
\begin{equation}
\Ycal_{\Mcal,\alphabold}^{\Qrm,\ell} 
= 
\left\{\Psibold\,:\, \Psibold \in \CCCcomp{\Qrm}[\Omega,\Tcal^
{\ell+\Qrm}(\RR^d)],\, \norm{\div_{\Mcal}^j\Psibold}_\infty\leq \alpha_j,\forall\,j=0,\dots\Qrm-1\right\},
\end{equation}
with $\div^j_\Mcal$ the $\Mcal$-anisotropic divergence operator of order $j$, see \Cref{sec: higher-order tdv,sec: BDV} for the precise definitions.
The higher-order total directional variation extends the classical notion of isotropic total generalized variation to the (smooth) elliptic anisotropic case.

\subsection{Related works}
The use of modified total variation regularisers in imaging processing has increased in the last decades, with the aim to enhance the local information in images.
We refer to the introduction of the complementary part of this work \cite{ParMasSch18applied} for a detailed review.

For our purposes it is useful to recall the \emph{total generalized variation} \cite{BreKunPoc2010,BreHol2014,Bredies2012} which appears in many image processing tasks. It is defined for a derivative order $\Qrm\geq 1$ as:
\begin{equation}
\TGV{\alphabold}{\Qrm,\ell}[\ubold] 
=\sup\left\{ \int_\Omega \ubold \cdot\div^\Qrm \Psibold \diff\xbold\,\Big\lvert\,
\begin{aligned}
&\Psibold\in\CCCcomp{\Qrm}[\Omega,\Sym{\ell+\Qrm}(\RR^d)],
\\
&\norm{\div^j \Psibold}_\infty\leq \alpha_j, \,\forall\, j=0,\dots,\Qrm-1
\end{aligned}\right\},
\label{eq: TDV intro theory}
\end{equation}
where $\Sym{\ell+\Qrm}(\RR^d)$ is the space of symmetric tensors,  $\alphabold=(\alpha_0,\dots,\alpha_{\Qrm-1})$ is a weight vector of positive real numbers, $\div\Psibold = \tr(\grad \otimes \Psibold)$ and $\div^j\Psibold = \tr^j(\grad^j \otimes \Psibold)$, \cite[Equation (2.1)]{BreKunPoc2010}.

In \cite{directionaltv}, the directional version of \cref{eq: TDV intro theory} is presented for a fixed and single global direction only and for an imaging function $u:\Omega\to\RR$: there, the continuous \emph{directional total variation} ($\DTV{}$) and \emph{directional total generalized variation} ($\DTGV{}{}$) are defined as:
\begin{align}
\DTV{}[u] 
&= 
\sup \left\{\int_\Omega u\div\widetilde{\Psibold}\diff \xbold\,\,\, \Psibold\in\CCCcomp{1}[\Omega,\RR^2],\, \widetilde{\Psibold}(\xbold)\in E^{a,\theta}(\bm{0}),\forall \xbold\in\Omega\right\},
\label{eq: DTV}
\\
\DTGV{\Qrm}{\alphabold}[u] 
&=
\sup\left\{ \int_\Omega u\div^\Qrm\widetilde{\Psibold}\diff \xbold\,\Big\lvert\, 
\begin{aligned}
&\Psibold \in \CCCcomp{\Qrm}[\Omega,\Sym{\Qrm}(\RR^2)],\, \widetilde{\Psibold}(\xbold)\in E^{a,\theta}(\bm{0}).\\
&\norm{\div^q{\widetilde{\Psibold}}}_\infty\leq \alpha_q,\,\forall\, q=0,\dots,\Qrm-1
\end{aligned}
\right\},
\label{eq: DTGV}
\end{align}
where $\widetilde{\Psibold}(\xbold) = \Rbold_\theta\Lambdabold_a\Psibold(\xbold)$ for $\Psibold\in B_1(\bm{0})$ and $a\in(0,1]$, with $\Rbold_\theta$ a rotation matrix and $\Lambdabold_a = \diag(1, a)$ a contraction matrix, and $\widetilde{\Psibold}(\xbold)\in E^{a,\theta}(\bm{0})$ where $E^{a,\theta}(\bm{0})$ is the closed elliptical set defined as $E^{a,\theta}(\bm{0})=\{\xbold\in\RR^2:\xbold^\T \Rbold_\theta \Lambdabold^2_{1/a}\Rbold_\theta^\T \xbold\leq 1\}$.

\subsection{Motivation of the paper}
We are interested in the analysis of the regulariser proposed in \cite{ParMasSch18applied} that generalises \eqref{eq: DTV}-\eqref{eq: DTGV} for handling non constant smoothing directions in the domain $\Omega\subset\RR^2$.  
In particular, we study the total directional variation $\TDVM{\Qrm,\ell}{\alphabold}[\ubold,\Mcal]$ (for a fixed order $\Qrm$ and a collection of weighting fields $\Mcal$) of a $\Tcal^\ell(\RR^d)$-valued function $u$. We analyse the space $\BDVM{\Qrm}{}$ of $\Tcal^\ell(\RR^d)$-valued functions whose total directional variation $\TDVM{\Qrm,\ell}{\alphabold}[\ubold,\Mcal]$ is finite. We exhibit an equivalent representation for $\TDVM{\Qrm,\ell}{\alphabold}[\ubold,\Mcal]$, and we prove the existence of solutions to the $\TDVM{\Qrm,\ell}{\alphabold}-\LL{2}$ problem.

We show that the theoretical results for $\TGV{\alphabold}{\Qrm,\ell}$, shown in \cite{BreKunPoc2010,Bredies2012,BreHol2014} for symmetric tensor fields and  isotropic  derivative operators naturally  extend to the case of possibly non-symmetric tensor fields and elliptic anisotropic derivative operators. A key for this extension is provided by Lemmas~\ref{lem: move M} and~\ref{lem: divergence modified}, and by Definition~\ref{eq: weighted div operator} which gives a suitable notion of weighted divergence for possibly non-symmetric tensor fields.

\subsection{Organization of the paper}
The paper is organized as follows: we introduce the preliminary notation in \cref{sec: preliminaries} and the higher-order total directional variation regularisers in \cref{sec: higher-order tdv}; in \cref{sec: BDV} we discuss the space of functions of bounded directional variation; in \cref{sec: equivalent} we show the equivalent decomposition of $\TDVM{\Qrm}{\alphabold}$, with respect to a collection of fields $\Mcal$ and in \cref{sec: existence} we prove the existence of solutions for the $\TDVM{\Qrm}{\alphabold}-\LL{2}$ problem. 
\section{Preliminaries}\label{sec: preliminaries}
In this section we introduce the notation of tensors and function spaces considered for the definition and analysis of $\TDVM{}{}$.

\subsection{Tensors}
Following \cite{BreKunPoc2010}, let $\Tcal^\ell(\RR^d)$ and $\Sym{\ell}[\RR^d]$ be the vector spaces of $\ell$-tensor and symmetric $\ell$-tensors in $\RR^d$, respectively defined as
\begin{align*}
\Tcal^\ell(\RR^d) 
&:=
\left\{\xibold:\underbrace{\RR^d\times \dots \times \RR^d}_{\ell\text{-times}}\to\RR,\text{ such that } \xibold \text{ is } \ell\text{-linear}\right\},
\\
\Sym{\ell}[\RR^d] &:=
\left\{ \xibold:
\underbrace{\RR^d \times \dots \times \RR^d}_{\ell\text{-times}}\to\RR,\text{ such that } \xibold \text{ is } \ell\text{-linear and symmetric} \right\},
\end{align*}
where $\xibold\in\Tcal^\ell(\RR^d)$ is symmetric if
$
\xibold(\abold_1,\dots,\abold_\ell) = \xibold(\abold_{\pi(1)},\dots,\abold_{\pi(\ell)})
$
for all permutations $\pi$ of $\{1,\dots,\ell\}$. By convention, $\Tcal^0(\RR^d)=\Sym{0}[\RR^d]$ is identified with $\RR$, and every element of $\Tcal^1(\RR^d)=\Sym{1}[\RR^d]$ can be identified with a vector of $\RR^d$ acting on $\RR^d$ through the scalar product. We have $\Tcal^\ell(\RR^d) \equiv \Sym{\ell}[\RR^d]$ only for $\ell=0,1$. For example, $\Tcal^2(\RR^d)$ can be identified with the space of general $d\times d$ real matrices, whereas $\Sym{2}(\RR^d)$ can be identified with the space of symmetric $d\times d$ real matrices.
We have the following operations on $\Tcal^\ell(\RR^d)$ (assuming that $a_k\in\RR^d$, $\forall k$):
\begin{itemize}
\item $\otimes$ is the tensor product for $\xibold_1\in\Tcal^{\ell_1}(\RR^d)$, $\xibold_2\in\Tcal^{\ell_2}(\RR^d)$, with $\xibold_1\otimes\xibold_2\in\Tcal^{\ell_1+\ell_2}(\RR^d)$:
\[
(\xibold_1\otimes\xibold_2)(\abold_1,\dots,\abold_{\ell_1 + \ell_2}) = \xibold_1(\abold_1,\dots,\abold_{\ell_1})\xibold_2(\abold_{\ell_1+1},\dots,\abold_{\ell_1+\ell_2});
\]
\item $\tr(\xibold)\in\Tcal^{\ell-2}(\RR^d)$ is the trace of $\xibold\in\Tcal^\ell(\RR^d)$, with $\ell\geq 2$, defined by
\[
\tr(\xibold)(\abold_1,\dots,\abold_{\ell-2}) = \sum_{i=1}^d \xibold(\ebold_i,\abold_1,\dots,\abold_{\ell-2},\ebold_i),
\]
where $\ebold_i$ is the $i$-th standard basis vector;
\item $(\blank)^\sim$ is the operator such that if $\xibold\in\Tcal^\ell(\RR^d)$, then
\[
\xibold^\sim(\abold_1,\dots \abold_\ell) = \xibold(\abold_\ell, \abold_1,\dots,\abold_{\ell-1});
\]
\item $\overline{(\blank)}$ is the operator such that if $\xibold\in\Tcal^\ell(\RR^d)$, then
\[
\overline{\xibold}(\abold_1,\dots \abold_\ell) = \xibold(\abold_\ell,\dots,\abold_1);
\]
\item let $\xibold,\etabold\in\Tcal^\ell(\RR^d)$. Then $\Tcal^\ell(\RR^d)$ is equipped with the scalar product defined as
\[
\xibold\cdot\etabold = \sum_{p\in\{1,\dots,d\}^\ell} \xibold(\ebold_{p_1},\dots,\ebold_{p_\ell}) \etabold(\ebold_{p_1},\dots,\ebold_{p_\ell});
\]
\item a Frobenius-type norm for $\xibold\in\Tcal^\ell(\RR^d)$ is given by $\abs{\xibold} = \sqrt{\xibold \cdot\xibold}$.
\end{itemize}

\subsection{Spaces}
Let $\Omega\subset\RR^d$ be a fixed open domain. We define the Lebesgue spaces of $\Tcal^\ell(\RR^d)$-valued tensor fields as
\[
\LL{p}[\Omega,\Tcal^\ell(\RR^d)]
=
\left\{
\xibold:\Omega\to\Tcal^\ell(\RR^d)\text{ is measurable and } \norm{\xibold}_p < \infty
\right\},
\]
with
\[
\norm{\xibold}_p 
=
\left(
\int_\Omega
\abs{\xibold(\xbold)}^p \diff\xbold
\right)^{\frac{1}{p}}
\text{ for } 1\leq p < \infty
\quad
\text{and}
\quad
\norm{\xibold}_\infty = \esssup_{\xbold\in\Omega}
\abs{\xibold(\xbold)}.
\]
Also $\LLloc{p}[\Omega,\Tcal^\ell(\RR^d)]$ is defined as usual: since the vector norm in $\Tcal^\ell(\RR^d)$ is a scalar product, then the duality holds: $\LL{p}[\Omega,\Tcal^\ell(\RR^d)]^\ast=\LL{p^\ast}[\Omega,\Tcal^\ell(\RR^d)]$, with $1/p+1/p^\ast = 1$ for $1\leq p < \infty$.

We now introduce the derivative for tensors and its weighted version. In the next, the elements of $\xibold:\Omega\to\Tcal^\ell(\RR^d)$ are described via the shortened notation
\[
(\xibold(\blank))_{i_1,\dots,i_\ell} := \xibold(\blank)(\ebold_{i_1},\dots,\ebold_{i_\ell}).
\]
\begin{definition}\label{def: derivative tensors}
Let $\grad=(\partial_1,\dots, \partial_d)^\T$ be the derivative operator and $\xibold:\Omega\to\Tcal^\ell(\RR^d)$ a differentiable tensor-valued function.
The $\Qrm$-th order (unweighted) derivative of $\xibold$ is defined as $(\grad^\Qrm \otimes \xibold):\Omega\to\Tcal^{\ell+\Qrm}(\RR^d)$ with
\[
\left( (\grad^\Qrm\otimes \xibold) (\blank)\right)_{j_1,\dots,j_\Qrm,i_1,\dots, i_\ell}
=
\left((\Drm^\Qrm\xibold(\blank))_{j_1,\dots,j_\Qrm}\right)_{i_1,\dots, i_\ell},
\]
where $\Drm^\Qrm\xibold:\Omega\to\Lcal^\Qrm(\RR^d,\Tcal^\Qrm(\RR^d))$ denotes the Fr\'{e}chet derivative of $\xibold$ and $\Lcal^\Qrm(\RR^d,\Tcal^\Qrm(\RR^d))$ the space of $\Qrm$-linear and continuous mappings from $\RR^\Qrm$ onto $\Tcal^\ell(\RR^d)$.
\end{definition}
\begin{definition}\label{def: weighted derivative tensors}
Let $\xibold, \grad, \Drm$ be as above and $\etabold:\Omega\to\Tcal^2(\RR^d)$. For $\Qrm=1$, the derivative operator weighted by $\etabold$ is defined as:
\[
\etabold\grad=\left(\sum_{k=1}^d \eta_{j,k}\partial_k\right)_{j=1}^d
\] 
and the first order derivative of $\xibold$ weighted by $\etabold$ is defined as $(\etabold\grad\otimes \xibold):\Omega\to\Tcal^{\ell+1}(\RR^d)$, with
\[
\left((\etabold\grad\otimes \xibold)(\blank)\right)_{j,i_1,\dots,i_\ell} = 
\left(
\left(
\etabold \Drm \xibold(\blank)
\right)_j
\right)_{i_1,\dots, i_\ell}.
\]
For the $\Qrm$-th order case, i.e.\ whenever each derivative order is weighted by the corresponding element of a collection $(\etabold^q)_{q=1}^\Qrm$, with each $\etabold^q:\Omega\to\Tcal^2(\RR^d)$, then
\[
\left((\etabold^\Qrm\grad\otimes\dots\otimes\etabold^1\grad\otimes \xibold)(\blank)\right)_{j_1,\dots,j_\Qrm,i_1,\dots,i_\ell} = 
\left(
\left(
\etabold^\Qrm \Drm
\left(\dots
\left(\etabold^1 \Drm \xibold(\blank)\right)\right)
\right)_{j_1,\dots,j_\Qrm}
\right)_{i_1,\dots, i_\ell}.
\]
\end{definition}

We denote the Banach space of $\Qrm$-times continuously differentiable $\Tcal^\ell(\RR^d)$-valued tensor fields as $\CCCspace{\Qrm}[\Omega,\Tcal^\ell(\RR^d)]$,
where $(\grad^\Qrm\otimes \ubold): \Omega\to\Tcal^{\Qrm+\ell}(\RR^d)$ and
\[
\norm{\ubold}_{\infty,\Qrm}
=
\max_{\ell=0,\dots,\Qrm} 
\sup_{\xbold\in\Omega}
\abs{\grad^\ell\otimes \ubold(\xbold)}.
\]
The space of fields in $\CCCspace{\Qrm}[\Omega,\Tcal^\ell(\RR^d)]$ with compact support  is denoted by $\CCCcomp{\Qrm}[\Omega,\Tcal^\ell(\RR^d)]$ and its completion under the supremum norm by $\CCospace{\Qrm}[\Omega,\Tcal^\ell(\RR^d)]$.
The space of Radon measures on $\Omega\subset\RR^d$ is denoted  by $\Mvarcal$ and, by Riesz representation  theorem, we identify:
$$\Mvarcal(\Omega,\Tcal^\ell(\RR^d))\equiv\CCospace{}[\Omega,\Tcal^\ell(\RR^d)]^\ast,$$
and we have \[\norm{\blank}_\Mvarcal =
\sup
\left\{
\langle \blank,\,\Psibold\rangle\,\Big\lvert\, \Psibold\in\CCospace{}[\Omega,\Tcal^\ell(\RR^d)],\,\norm{\Psibold}_\infty\leq 1
\right\}.
\]
$\Dcal'(\Omega,\Tcal^\ell(\RR^d))$ denotes the space of $\Tcal^\ell(\RR^d)$-valued distributions on $\Omega$ and $\Dcal(\Omega,\Tcal^\ell(\RR^d))=\CCCcomp{\infty}[\Omega,\Tcal^\ell(\RR^d)]$  is the associated space of test functions.

\subsection{Notation}
In what follows, we deal with derivatives of order up to $\Qrm\in\NN^\ast$. 
Since the weighting of each derivative order is the core operation of this work, we make use of a collection of smooth weighting tensor fields $\Mcal=(\Mbold_j)_{j=1}^\Qrm$, where for all $j\in\{1,\dots, \Qrm\}$, $\Mbold_j:\Omega\to\Tcal^2(\RR^d)$ and $\forall \xbold\in\Omega$, $\Mbold_j(\xbold)$ can be identified with a positive definite $d\times d$ matrix. When $\Qrm=1$ or when only one derivative is involved, we will refer directly to a unique weighting tensor field $\Mbold$. 
\section{Higher-order total directional variation}\label{sec: higher-order tdv}
For making sense of the distributional formulation of higher-order directional variation in \cref{eq: TDV intro def} we need an integration by parts formula for the weighted derivative of tensors in \cref{def: weighted derivative tensors}. Namely we consider
\[
\int_\Omega (\Mbold\grad\otimes \Abold)\cdot\Psibold \diff \xbold,
\] 
with $\Omega\subset\RR^d$  a bounded Lipschitz domain, $\Mbold\in\CCCspace{1}[\Omega,\Tcal^2(\RR^d)]$, $\Abold\in\CCCspace{1}[\Omega,\Tcal^\ell(\RR^d)]$ and $\Psibold\in\CCCspace{1}[\Omega,\Tcal^{\ell+1}(\RR^d)]$.
We immediately explore the transfer of $\Mbold$ on $\Psibold$:
\begin{lemma}\label{lem: move M}
Let $\Omega$, $\Mbold$, $\Abold$ and $\Psibold$ as above. Then:
\begin{equation}
\int_\Omega (\Mbold\grad\otimes\Abold)\cdot\Psibold \diff\xbold =\int_\Omega (\grad\otimes \Abold)\cdot \tr\left(\Mbold\otimes\Psibold^\sim\right)\diff\xbold,\,\text{ for all }\Mbold,\,\Abold,\,\Psibold.
\end{equation}
\end{lemma}
\begin{proof}
Using Einstein notation, we have:
\[
\begin{aligned}
(\Mbold\grad \otimes\Abold)\cdot\Psibold  
&= 
\Mbold_{j,k} \partial_k \Abold_{i_1,\dots, i_\ell} 
\Psibold_{j,i_1,\dots,i_\ell}\\
&= 
\partial_k \Abold_{i_1,\dots, i_\ell}
\Mbold_{j,k} \Psibold_{j,i_1,\dots,i_\ell}
\\
&= 
\left(\grad\otimes\Abold\right)_{k,i_1,\dots, i_\ell}
\Mbold_{j,k} (\Psibold^\sim)_{i_1,\dots,i_{\ell},j}\\
&= 
\left(\grad\otimes\Abold\right)_{k,i_1,\dots, i_\ell}
(\Mbold\otimes\Psibold^\sim)_{j,k,i_1,\dots,i_\ell,j}\\
&= 
\left(\grad\otimes\Abold\right)_{k,i_1,\dots, i_\ell}
\left(\tr(\Mbold\otimes\Psibold^\sim)\right)_{k,i_{1},\dots,i_{\ell}}
\\
&= 
(\grad\otimes \Abold)\cdot \tr(\Mbold\otimes\Psibold^\sim).
 \end{aligned}
 \]
Therefore we get
\[
\int_\Omega (\Mbold\grad\otimes\Abold)\cdot\Psibold \diff\xbold =\int_\Omega (\grad\otimes \Abold)\cdot \tr\left(\Mbold\otimes\Psibold^\sim\right)\diff\xbold.
\]
\end{proof}

\begin{definition}
Let $\Mbold$ and $\Psibold$ as above. We define the $\Mbold$--divergence of $\Psibold$ as:
\begin{equation}
\div_\Mbold (\Psibold) = \tr\left(\grad\otimes \left[\tr(\Mbold\otimes\Psibold^\sim)\right]^\sim\right).
\label{eq: weighted div operator}
\end{equation}
\end{definition}

\begin{remark}
For $\Mbold=\Ibold$ the divergence in \cref{eq: weighted div operator} is
$
\div(\Psibold) = \tr(\grad\otimes\Psibold^\sim).
$
When $\Psibold$ is a symmetric tensor, since $\Psibold^\sim=\Psibold$, we retrieve
$
\div(\Psibold) = \tr(\grad\otimes{\Psibold})
$
of \cite[Equation (2.1)]{BreKunPoc2010}.
\end{remark}

The next lemma provides an integration by parts formula which justifies the definition of the $\Mbold$--divergence operator.
\begin{lemma}\label{lem: divergence modified}
Let $\Omega$, $\Mbold$, $\Abold$ and $\Psibold$ as above. Then:
\begin{equation}
\int_\Omega (\Mbold\grad\otimes \Abold)\cdot\Psibold \diff \xbold 
= 
\int_{\partial \Omega}
(\nu\otimes\Abold)
\cdot 
\tr(\Mbold\otimes\Psibold^\sim)
\diff \Hcal^{d-1}
-\int_\Omega \Abold \cdot \div_{\Mbold}\Psibold\diff\xbold,
\label{eq: integration by parts}
\end{equation}
where $\nu$ is the outward unit normal on $\partial \Omega$ and $
\div_{\Mbold}\Psibold$ as in \cref{eq: weighted div operator}.
\end{lemma}
\begin{proof}
We know from \cref{lem: move M} that:
\[
\int_\Omega (\Mbold\grad\otimes\Abold)\cdot\Psibold \diff\xbold 
=
\int_\Omega (\grad\otimes \Abold)\cdot \tr\left(\Mbold\otimes\Psibold^\sim\right)\diff\xbold.
\]
Let $\Phibold:=\tr\left(\Mbold\otimes\Psibold^\sim\right)\in\Tcal^{\ell+1}(\RR^d)$.
From Gauss-Green theorem, in Einstein notation:
\[
\begin{aligned}
\int_\Omega 
(\grad\otimes \Abold)\cdot\Phibold
\diff\xbold 
&= 
\int_\Omega 
\partial_{k} \Abold_{i_1,\dots, i_\ell}\Phibold_{k,i_1,\dots, i_\ell} 
\diff \xbold
\\
&=
\int_{\partial \Omega}
\nu_k \Abold_{i_1,\dots,i_\ell} \Phibold_{k,i_1,\dots,i_\ell}
\diff \Hcal^{d-1}
-\int_\Omega 
\Abold_{i_1,\dots, i_\ell}   \partial_{k}  (\Phibold^\sim)_{i_1,\dots, i_\ell, k} 
\diff\xbold.
\end{aligned}
\]
Now, by remarking that
\[
\nu_k 
\Abold_{i_1,\dots,i_\ell}
=
(\nu \otimes \Abold)_{k,i_1,\dots,i_\ell}
\]
and
\[
\partial_{k}  (\Phibold^\sim)_{i_1,\dots, i_\ell,k} 
= 
(\grad\otimes\Phibold^\sim)_{k,i_1,\dots, i_\ell, k} 
= 
\left(
\tr(\grad\otimes\Phibold^\sim)
\right)_{i_1, \dots, i_\ell},
\]
we conclude
\[
\begin{aligned}
\int_\Omega (\Mbold\grad\otimes\Abold)\cdot \Psibold\diff\xbold 
=
&
\int_{\partial \Omega}
(\nu\otimes\Abold) \cdot \tr(\Mbold\otimes\Psibold^\sim)
\diff \Hcal^{d-1}\\
&
-\int_\Omega \Abold \cdot \tr(\grad\otimes\left[\tr(\Mbold\otimes\Psibold^\sim)\right]^\sim) \diff\xbold.
\end{aligned}
\]
\end{proof}

\begin{remark}
For $\Psibold\in\CCCcomp{1}[\Omega,\Tcal^\ell(\RR^d)]$, in \cref{lem: divergence modified} the integral on $\partial \Omega$ vanishes:
\[
\int_\Omega (\Mbold\grad\otimes \Abold)\cdot\Psibold \diff \xbold 
= 
-\int_\Omega \Abold \cdot \div_{\Mbold}\Psibold\diff\xbold.
\]
\end{remark}

With the notion of weighted divergence $\div_\Mbold$ of a $(\ell+1)$-tensor field in place, we can talk about \emph{weak derivatives}, similarly to \cite[Definition 2.4]{Bredies2012}.

\begin{definition}
Let $\Mbold\in\CCCspace{1}[\Omega,\Tcal^2(\RR^d)]$. We say that $\Abold\in\LLloc{1}[\Omega,\Tcal^\ell(\RR^d)]$ has a weak $\Mbold$-weighted derivative if there exists $\etabold\in\LLloc{1}[\Omega,\Tcal^\ell(\RR^d)]$ such that
\[
\int_\Omega \etabold\cdot \Psibold\diff\xbold = -\int_\Omega \Abold\cdot\div_\Mbold\Psibold\diff\xbold
\]
for all $\Psibold\in\CCCcomp{1}[\Omega,\Tcal^\ell(\RR^d)]$. We write $\etabold = \Mbold\grad\otimes\Abold$ in this case.
\end{definition}

We can now define the \emph{total directional variation} of order $\Qrm$ for $\ubold\in\LL{1}[\Omega,\Tcal^\ell(\RR^d)]$.

\begin{definition}\label{def: TDV}
Let $\Omega\subset\RR^d$, $\ubold\in\LL{1}[\Omega,\Tcal^\ell(\RR^d)]$, $\Qrm\in\NN$, $\Mcal := (\Mbold_j)_{j=1}^\Qrm$ be a collection of fields in $\CCCspace{\infty}[\Omega,\Tcal^2(\RR^d)]$ and $\alphabold:=(\alpha_0,\dots,\alpha_{\Qrm-1})$ be a positive weight vector.  The total directional variation of order $\Qrm$, associated with $\Mcal$ and $\alphabold$, is defined as:
\begin{equation}
\TDVM{\Qrm,\ell}{\alphabold}[\ubold,\Mcal]:= \sup_\Psibold 
\left\{ 
\int_\Omega \ubold\cdot \div_{\Mcal}^\Qrm \Psibold \diff \xbold \,\Big\lvert\,\text{for all } \Psibold \in\Ycal_{\Mcal,\alphabold}^{\Qrm,\ell} \right\},
\label{eq: TDV definition M}
\end{equation}
where
\begin{equation}
\Ycal_{\Mcal,\alphabold}^{\Qrm,\ell}
= 
\left\{
\Psibold\,:\, \Psibold \in \CCCcomp{\Qrm}[\Omega,\Tcal^{\ell+\Qrm}(\RR^d)],\, \norm{\div_{\Mcal}^j\Psibold}_\infty\leq \alpha_j,\forall\,j=0,\dots\Qrm-1
\right\}
\label{eq: Ycal}
\end{equation}
and the weighted divergence of order $j\in[0,\Qrm]$ is defined recursively, from \cref{lem: divergence modified}, as:
\begin{equation}
\begin{aligned}
\div^0_{(\blank)}(\Psibold) &:=\Psibold, &\text{if } j&=0,\\
\div^1_{(\Mbold_\Qrm)}(\Psibold) &:= \tr\left(\grad\otimes \left[\tr(\Mbold_\Qrm\otimes\Psibold^\sim)\right]^\sim\right), &\text{if } j&=1,\\
\div_{(\Mbold_{\Qrm-j+1},\dots,\Mbold_{\Qrm)}}^j(\Psibold) &:=  
\div_{(\Mbold_{\Qrm-j+1})}^1
\left(
\div_{(\Mbold_{\Qrm-j+2},\dots,\Mbold_{\Qrm})}^{j-1} (\Psibold)
\right) &\text{if } j&=2,\dots,\Qrm.
\end{aligned}
\end{equation}
Thus the $\Qrm^\text{th}$ weighted divergence w.r.t.\ $\Mcal$ is $
\div^\Qrm_{\Mcal}(\Psibold) := 
\div_{\Mbold_1}^1\left(
\div_{\Mbold_2}^1\left(
\dots\left( 
\div_{\Mbold_\Qrm}^1 (\Psibold)
\right)
\right)
\right).$
\end{definition}

\begin{remark}
For $\Mcal=(\Ibold)_{j=1}^{\Qrm}$, where $\Ibold$ is the identity matrix, then $\TDVM{\Qrm,\ell}{\alphabold}[\ubold,\Mcal]$ coincides with extension to $\Tcal^\ell(\RR^d)$ tensors of the non-symmetric total generalized variation $\neg\mathrm{sym}\TGV{\alphabold}{\Qrm,\ell}[\ubold]$ defined (for $\ell=0$) in \cite[Remark 3.10]{BreKunPoc2010}.
\end{remark}
 
\section{Tensor fields of bounded directional variation}\label{sec: BDV}
In what follows,
we introduce 
the space of bounded directional variation $\BDVM{\Qrm}{}[\Omega,\Mcal,\Tcal^\ell(\RR^d)]$, which is the natural space for the $\TDVM{}{}$ regulariser. 
We also state some results about the kernel of the weighted derivatives. 
To do so, we will treat the discussion of these spaces for first- and higher-order derivatives, separately, so as to build a recursion rule for tensors of bounded directional variation with weighted derivatives of any order $\Qrm>0$.

\subsection{First order derivative}
As said, when $\Qrm=1$ then the collection $\Mcal$ is made by one smooth tensor field only, namely $\Mbold$: therefore we will use $\Mbold$ within this section. We will always assume that $\Mbold(\xbold)$ can be identified with a positive definite matrix at every point of $\Omega$. 

\begin{remark}
For $\Qrm=1$, when $\TDVM{1,\ell}{\alphabold}[\ubold]<\infty$ in \cref{eq: TDV definition M},  then
the weak weighted derivative is a Radon measure on $\Omega$ with values in $\Tcal^{\ell+1}(\RR^d)$.
\end{remark}

\begin{definition}\label{def: total directional variation order 1}
The \emph{total directional variation} of a $\Tcal^\ell(\RR^d)$-valued function $\ubold$ w.r.t.\ the field $\Mbold$ is defined as the Radon norm of $\Mbold\grad\otimes \ubold$ and indicated as:
\begin{equation}
\TDVM{1,\ell}{\bm{1}}[\ubold,\Mbold] 
= 
\norm{\Mbold\grad\otimes \ubold}_{\Mvarcal(\Omega,\Tcal^{\ell+1}(\RR^d))}.
\end{equation}
\end{definition}

\begin{definition}
Let $\Omega\subset\RR^d$ be a bounded Lipschitz domain and $\Mbold\in\CCCspace{\infty}[\Omega,\Tcal^2(\RR^d)]$ such that $\Mbold(\xbold)$ is a positive definite matrix at every $\xbold\in\Omega$.
The space of $\Tcal^\ell(\RR^d)$-valued tensor functions $\ubold$ of \emph{bounded directional variation of order $1$} with respect to the field $\Mbold$ is defined as 
\[
\BDVM{1}{}[\Omega,\Mbold,\Tcal^\ell(\RR^d)] 
= 
\left\{
\ubold\in\LL{1}[\Omega,\Tcal^\ell(\RR^d)]\,\Big\lvert\, \Mbold\grad\otimes \ubold\in\Mvarcal(\Omega,\Tcal^{\ell+1}(\RR^d))
\right\}.
\]
\end{definition}

For simplicity, we denote $\BDVM{}{}[\ubold,\Mbold,\Tcal^\ell(\RR^d)]=\BDVM{1}{}[\ubold,\Mbold,\Tcal^\ell(\RR^d)]$.

\begin{remark}
Since $\Tcal^\ell(\RR^d)\equiv\Sym{\ell}(\RR^d)$ for $\ell=0,1$, it is easily seen that
\[
\BDVM{}{}[\Omega,\Ibold,\Tcal^0(\RR^d)]  
\equiv
\BV{}[\Omega,\RR]\quad\text{and}\quad
\BDVM{}{}[\Omega,\Ibold,\Tcal^1(\RR^d)]  
\equiv
\BV{}[\Omega,\RR^d],
\]
with $\BV{}[\Omega,\RR]$, $\BV{}[\Omega,\RR^d]$ the spaces of scalar-valued  and vector-valued functions of bounded variation, respectively~\cite{AmbrosioBook2000}.
\end{remark}

We now prove that tensor fields of bounded directional variation can be approximated by smooth functions, similarly to \cite[Proposition 4.13]{Bredies2012}. For doing so we firstly need to show that the weighted gradient is closed, similarly to \cite[Proposition 4.2]{Bredies2012}.
\begin{proposition}\label{prop: closedness of weighted derivative}
Let $p\in[1,\infty]$. If $\ubold_j\weakconv\ubold$ in $\LL{p}[\Omega,\Tcal^\ell(\RR^d)]$ and $\Mbold\grad\otimes\ubold_j \weakconv \etabold$  in $\LL{p}[\Omega,\Tcal^{\ell+1}(\RR^d)]$, then $\Mbold\grad\otimes \ubold = \etabold$, i.e.\ the weighted gradient is closed in the distributional sense. The statement remains true for weak$-\ast$ convergence in $\Mvarcal(\Omega, \Tcal^\ell(\RR^d))$ and $\Mvarcal(\Omega,\Tcal^{\ell+1}(\RR^d))$, respectively.
\end{proposition}
\begin{proof}
Omitted since it is just a notational adaptation of \cite[Proposition 4.2]{Bredies2012}.
\end{proof}

Similarly to \cite[Proposition 4.13]{Bredies2012}, we can approximate functions of bounded directional deformation with smooth functions.
\begin{proposition}\label{prop: 4.13 Bredies2012}
Let $\Omega$ be a bounded domain. 
The set $\CCCspace{\infty}[\Omega,\Tcal^\ell(\RR^d)]\cap\BDVM{}{}[\Omega,\Mbold,\Tcal^\ell(\RR^d)]$ is dense in $\BDVM{}{}[\Omega,\Mbold,\Tcal^\ell(\RR^d)]$ in the sense that for each $\ubold\in\BDVM{}{}[\Omega,\Mbold,\Tcal^\ell(\RR^d)]$ there exists an approximating sequence $\{\ubold_j\}_{j\in\NN}\subset\CCCspace{\infty}[\Omega,\Mbold,\Tcal^\ell(\RR^d)]$ that converges strictly to $\ubold$, i.e.,
\begin{equation}
\begin{dcases}
\ubold_j \to \ubold &\text{in }\LL{1}[\Omega, \Mbold,\Tcal^\ell(\RR^d )],\\
\Mbold\grad\otimes \ubold_j\weakstarconv \Mbold\grad\otimes \ubold &\text{in }\Mvarcal(\Omega,\Tcal^\ell(\RR^d ),\\
\norm{\Mbold\grad\otimes \ubold_j}_\Mvarcal \to \norm{\Mbold\grad\otimes \ubold}_\Mvarcal.
\end{dcases}
\label{eq: convergence added}
\end{equation}
If the support of $\ubold$ is compact in $\Omega$, then $(\ubold_j)_{j\in\NN}$ can be chosen such that each $\ubold_j$ is in
$\CCCcomp{\infty}[\Omega,\Tcal^\ell(\RR^d)]$.
\end{proposition}
\begin{proof}
The proof is based on a standard use of mollifiers so as to obtain a sequence $(\ubold_j)_{j\in\NN}$ in $\CCCspace{\infty}[\Omega, \Tcal^\ell(\RR^d)]$ satisfying the first and the third convergence in \cref{eq: convergence added}.
The boundedness of $(\Mbold \grad\otimes\ubold_j)_{j\in\NN}$  in $\Mvarcal(\Omega,\Tcal^\ell(\RR^d))$ implies that there exists a subsequence (not relabelled) weakly-$\ast$ converging to $\Mbold\grad\otimes \ubold$ since the operator $\Mbold\grad$ is closed by \cref{prop: closedness of weighted derivative}. 
\end{proof}

We are now going to discuss some results about the kernel of the weighted derivative operator $\ker(\Mbold\grad)$: in order to do so, we will define a continuous projection map $R$ onto $\ker(\Mbold\grad)$, so as to prove the coercivity estimate for the total directional variation in \cref{eq: sobolev-korn 1}.

\begin{remark}\label{rem: 4.5}
Being $\ker(\Mbold\grad)$ the space of polynomials of vanishing first weighted derivative, it is in $\LL{\infty}[\Omega,\Tcal^\ell(\RR^d)]$ because $\Omega$ is bounded, therefore
\[
\ker(\Mbold\grad)^\perp 
= 
\left
\{
\fbold\in\LL{d}[\Omega,\Tcal^\ell(\RR^d)]\,\Big\lvert\, \int_\Omega \fbold\cdot \ubold=0\diff\xbold,\text{ for all }  \ubold\in\ker(\Mbold\grad)
\right\}
\]
is a closed subspace of $\LL{d}[\Omega, \Tcal^\ell(\RR^d)]$.
\end{remark}

\begin{remark}
Note also that $\ker(\Mbold\grad)\equiv\ker(\grad)$ since the field $\Mbold$ is assumed everywhere invertible.
\end{remark}

\begin{proposition}\label{prop: beginning appendix A bre-hol}
There exists a continuous projection $R:\LL{d}[\Omega,\Tcal^\ell(\RR^d)] \to \LL{d}[\Omega,\Tcal^\ell(\RR^d)]$ such that 
\[
\Im(R)=\ker(\Mbold\grad)=\ker(\grad)\quad\text{and}\quad \ker(R)=\ker(\Mbold\grad)^\perp=\ker(\grad)^\perp.
\]
\end{proposition}
\begin{proof}
The proof is an easy adaptation of the proof given at the beginning of \cite[Appendix A]{BreHol2014}. 
We observe that $\ker(\Mbold\grad)$ is finite-dimensional, therefore 
\[
\LL{d}[\Omega,\Tcal^\ell(\RR^d)] = \ker(\Mbold\grad) \oplus \ker(\Mbold\grad)^\perp
\]
and since both subspaces are closed, then the open mapping theorem implies that there exists a continuous projection $R$ such that:
\begin{equation}
R:\LL{d}[\Omega,\Tcal^\ell(\RR^d)] \to \LL{d}[\Omega,\Tcal^\ell(\RR^d)]
\label{eq: R map}
\end{equation}
with $\Im(R)=\ker(\Mbold\grad)$ and $\ker(R)=\ker(\Mbold\grad)^\perp$, see \cite[Example 1, pag.\ 38]{BrezisBook2010}.
As consequence, the adjoint projection $R^\ast$ is a continuous projection in $\LL{d/(d-1)}[\Omega,\Tcal^\ell(\RR^d)]$ onto $\ker(\Mbold\grad)^{\perp\perp}=\ker(\Mbold\grad)$.
\end{proof}

The following Sobolev-Korn inequality holds similarly to \cite[Corollary 4.20]{Bredies2012}, which will be proved for the general case $\Qrm\geq 1$ in \cref{prop: coercivity}.
\begin{lemma}
For any continuous projection $R$ onto $\ker(\Mbold\grad)$ as in \cref{eq: R map}, there exists a constant $C > 0$, depending only on $\Omega$, $R$ and $\Mbold^{-1}$, such that it holds for each $\ubold\in\BDVM{}{}[\Omega,\Mbold,\Tcal^\ell(\RR^d)]$:
\begin{equation}
\norm{\ubold-R \ubold}_{d/(d-1)} \leq C\norm{\Mbold\grad \otimes \ubold}_\Mvarcal.
\label{eq: sobolev-korn 1}
\end{equation}
\end{lemma}
\begin{proof}
We firstly need to prove $
\norm{\ubold - R \ubold}_1 \leq C\norm{\Mbold\grad \otimes \ubold}_\Mvarcal$. This follows by the same proof in \cite[Theorem 4.19]{Bredies2012} with minor notational changes.
From the continuous embedding of $\BDVM{}{}$ into $\LL{d/(d-1)}[\Omega,\Tcal^\ell(\RR^d)]$ proved later in \cref{th: bredies 4.16} (with \cref{th: bredies 4.16 estimate} in place) we have
\[
\norm{\ubold - R \ubold}_{d/(d-1)}
\leq
\widetilde{C}\left(
\norm{\ubold-R\ubold}_1 
+
\norm{\Mbold\grad\otimes \ubold}_{\Mvarcal}
\right)
\leq
C\norm{\Mbold\grad\otimes\ubold}_{\Mvarcal}.
\]
\end{proof}

\begin{definition} 
Let 
$
B_\varepsilon(\bm{0})=
\left\{
\xbold \in \RR^d\,:\, \norm{\xbold}_2 \leq \varepsilon
\right\}
$ 
be the $\ell^2$-closed $\varepsilon$-ball centred at $\bm{0}\in\RR^d$ and
$
B_{\Mbold,\varepsilon}(\bm{0})
=
\left\{
\ybold \in \RR^d\,:\, \norm{\Mbold \xbold}_2 \leq \varepsilon
\right\}
$ 
be the $\Mbold$-anisotropic closed $\varepsilon$-ball centred at $\bm{0}\in\RR^d$.
\end{definition}

Similarly to \cite[Lemma A.1]{BreHol2014}, we can prove the following lemma.
\begin{lemma}\label{lem: A1}
The closure of the set
\[
U = 
\left\{
-\div_\Mbold \Psibold\,\lvert\, \Psibold\in\CCCcomp{\infty}[\Omega,\Tcal^{\ell+1}(\RR^d)],\,\norm{\Psibold}_\infty\leq 1
\right\},
\]
in $\LL{d}[\Omega,\Tcal^\ell(\RR^d)]\cap \ker(\Mbold\grad)^\perp$ contains $\bm{0}$ as interior point.
\end{lemma}
\begin{proof}
We have to check the functional $F : \LL{d/(d-1)}[\Omega,\Tcal^\ell(\RR^d)]\to[0,\infty]$ is coercive:
\[
F(\ubold) = \norm{\Mbold\grad \otimes \ubold}_\Mvarcal + I_{\{\bm{0}\}}(R^\ast \ubold),
\]
where $R$ is the continuous projection map defined in \cref{eq: R map} and $I_Z$ is the indicator function of this set, i.e.\ $I_Z(\xbold) = 0$ if $\xbold\in Z$ and $I_Z(\xbold) = \infty$ otherwise.

Let $(\ubold_j)_j\in\LL{d/(d-1)}[\Omega,\Tcal^\ell(\RR^d)]$ with $\norm{\ubold_j}_{d/(d-1)}\to \infty$. 
We can distinguish two cases: either $F(\ubold_j) = \infty$ or $F(\ubold_j)<\infty$, which is the case for $\ubold_j\in \BDVM{}{}[\Omega,\Mbold,\Tcal^\ell(\RR^d)]\cap \ker(R^\ast)$. 

When $F(\ubold_j)<\infty$,
then $R^\ast \ubold_j = 0$ and the Sobolev-Korn inequality in \cref{eq: sobolev-korn 1} gives
\[
\norm{\ubold_j}_{d/(d-1)} \leq C \norm{\Mbold\grad \otimes \ubold_j}_\Mvarcal = C F(\ubold_j),
\]
for a constant $C>0$, independently of $j$. 
This means that $F(\ubold_j)\to \infty$ and the coercivity is proved. 
Thus, the Fenchel conjugate of $F$
\[
F^\ast:\LL{d}[\Omega,\Tcal^\ell(\RR^d)]\to ]-\infty,\infty]
\]
is continuous at $\bm{0}$ \cite[Theorem 4.4.10]{borwein2010convex}. Since $\ker(R^\ast)=\Im(\Ibold-R^\ast)$ we have
\[
\begin{aligned}
F^\ast(\Psibold) 
&= 
\sup_{\ubold\in\ker(R^\ast)} \langle \Psibold,\, \ubold \rangle - \norm{\Mbold\grad \otimes \ubold}_\Mvarcal \\
&=
\sup_{ \ubold\in \LL{d/(d-1)}[\Omega,\Tcal^\ell(\RR^d)]} \langle \Psibold,\,  \ubold- R^\ast  \ubold \rangle - \norm{\Mbold\grad \otimes ( \ubold-R^\ast  \ubold)}_\Mvarcal \\
&=
\sup_{\ubold \in \LL{d/(d-1)}[\Omega,\Tcal^\ell(\RR^d)]} 
\langle \Psibold - R\Psibold, \, \ubold \rangle - \norm{\Mbold\grad \otimes  \ubold}_\Mvarcal \\
&= I_U^{\ast\ast} (\Psibold -R\Psibold)\\
&= I_{\overline{U}}(\Psibold -R\Psibold).
\end{aligned}
\]
The continuity in $\bm{0}$ implies that there exists $\varepsilon>0$ such that the anisotropic ball $B_{\Mbold,\varepsilon}$ induced by $\Mbold$, is such that $B_{\Mbold,\varepsilon}(\bm{0})\subset (I-R)^{-1}(\overline{U})$. Thus, for each $\Psibold\in\LL{d}[\Omega,\Tcal^\ell(\RR^d)]\cap\ker(\Mbold\grad)^\perp$ with $\norm{\Psibold}_d \leq \varepsilon$, we have $\Psibold = \Psibold-R\Psibold\in \overline{U}$, showing that $\bm{0}$ is an interior point.
\end{proof}

We can now prove that a distribution $ \ubold$ is in $\BDVM{}{}[\Omega,\Mbold,\Tcal^\ell(\RR^d)]$ as soon as the weighted derivative $\Mbold\grad$ is a Radon measure, similarly to \cite[Theorem 2.6]{BreHol2014}.
\begin{theorem} \label{th: 2.6}
Let $\Omega\in\RR^d$ be a bounded Lipschitz domain and $\ubold\in\Dcal'(\Omega,\Tcal^\ell(\RR^d))$ be a distribution such that $\Mbold\grad \otimes \ubold\in\Mvarcal\left(\Omega,\Tcal^{\ell+1}(\RR^d)\right)$ in the distributional sense, 
for a positive definite field $\Mbold\in\CCCspace{\infty}[\Omega,\Tcal^2(\RR^d)]$.  Then, $\ubold\in\BDVM{}{}[\Omega,\Mbold,\Tcal^\ell(\RR^d)]$.
\end{theorem}
\begin{proof}
Let $\ubold\in\Dcal'(\Omega,\Tcal^\ell(\RR^d))$ be such that $\Mbold\grad \otimes \ubold \in \Mvarcal\left(\Omega,\Tcal^{\ell+1}(\RR^d)\right)$ in the distributional sense. 
We need to prove that $\ubold\in\LL{1}[\Omega,\Tcal^{\ell+1}(\RR^d)]$.

Let $X=\LL{d}[\Omega,\Tcal^\ell(\RR^d)]\cap\ker(\Mbold\grad)^\perp$, which is a Banach space with the induced norm. 
Let $\delta>0$ and $U$ from \cref{lem: A1}, such that $B_{\Mbold,\delta}(\bm{0})$ exists and
$
B_{\Mbold,\delta}(\bm{0})
\subset 
\overline{U}
\subset X
$.
We define also the following sets:
\[
\begin{aligned}
K_1 
&= 
\left\{
\Psibold \in \CCCcomp{\infty}[\Omega,\Tcal^{\ell+1}(\RR^d)]\,\lvert\, \norm{\Psibold}_\infty\leq \delta^{-1},\,\norm{-\div_\Mbold\Psibold}_d\leq 1
\right\},\\
K_2 &= 
\left\{
\Psibold \in \CCCcomp{\infty}[\Omega,\Tcal^{\ell+1}(\RR^d)]\,\lvert\, \norm{\Psibold}_\infty\leq \delta^{-1}
\right\}.
\end{aligned}
\]
Straightforwardly, we have $K_1\subset K_2$. By testing $\ubold$ with $-\div_\Mbold\Psibold$ and $\Psibold \in K_1$, since $\Mbold\grad \otimes \ubold\in \Mvarcal\left(\Omega,\Tcal^{\ell+1}(\RR^d)\right)$, we get by density
\[
\begin{aligned}
\sup_{\Psibold\in K_1} \langle \ubold, -\div_\Mbold\Psibold \rangle
&
\leq
\sup_{\Psibold\in K_2} \langle \ubold, -\div_\Mbold\Psibold \rangle
=
\sup_{\substack{\Psibold\in \CCospace{}[\Omega,\Tcal^{\ell+1}(\RR^d)]\\ \norm{\Psibold}_\infty\leq\delta^{-1}}} \langle \Mbold\grad \otimes \ubold, \Psibold \rangle 
\\
&
=
\delta^{-1}\norm{\Mbold\grad \otimes \ubold}_{\Mvarcal}.
\end{aligned}
\]
One  can  show  that
$
\overline{\{-\div_\Mbold \Psibold \,\lvert\, \Psibold\in K_1 \}} = \overline{B_{\Mbold,1}(\bm{0})}\in X
$
and thus
\[
\sup_{\Psibold\in K_1} \langle \ubold, -\div_\Mbold\Psibold \rangle = \norm{\ubold}_{X^\ast},
\]
i.e.\ $\ubold$ can be extended to an element in $X^\ast$. 
Also, $X$ is a closed subspace of $\LL{d}[\Omega,\Tcal^\ell(\RR^d)]$ and by Hahn-Banach theorem $\ubold$ can be extended to $\vbold\in\LL{d}[\Omega,\Tcal^\ell(\RR^d)]^\ast =\LL{d/(d-1)}[\Omega,\Tcal^\ell(\RR^d)]$. 
Thus $\vbold\in\LL{1}[\Omega,\Tcal^\ell(\RR^d)]$ with the distribution $\ubold-\vbold\in\ker(\Mbold\grad)$ and we have
\[
\langle \ubold-\vbold,\,-\div_\Mbold\Psibold\rangle = \langle \ubold-\ubold,\,-\div_\Mbold\Psibold\rangle = \bm{0},
\quad\text{for each }\quad
\Psibold\in\CCCcomp{\infty}[\Omega,\Tcal^{\ell+1}(\RR^d)],
\]  
since $-\div_\Mbold \Psibold\in X$: so $\ubold-\vbold$ is a polynomial of degree less than $\ell$, $(\ubold-\vbold)\in\LL{1}[\Omega,\Tcal^\ell(\RR^d)]$ and $\ubold=\vbold+(\ubold-\vbold)\in\LL{1}[\Omega,\Tcal^\ell(\RR^d)]$, leading to $\ubold\in\BDVM{}{}[\Omega,\Mbold,\Tcal^\ell(\RR^d)]$.
\end{proof}

\subsection{Higher-order derivatives}
When $\Qrm$ order of derivatives are involved, then we deal with the collection of tensor fields $\Mcal=(\Mbold_j)_{j=1}^\Qrm$.
For a distribution $\ubold\in\Dcal'\left(\Omega,\Tcal^\ell(\RR^d)\right)$ we get from \cref{th: 2.6}:
\[
\Mbold_\Qrm\grad\otimes\dots\otimes\Mbold_1\grad \otimes \ubold 
\in
\Mvarcal(\Omega,\Tcal^{\ell+\Qrm}(\RR^d)),
\]
which implies
\[
\Mbold_{\Qrm-1}\grad\otimes\dots\otimes\Mbold_1\grad \otimes \ubold\in\BDVM{}{}[\Omega,\Mbold_\Qrm,\Tcal^{\ell+\Qrm-1}(\RR^d)],
\]
thus we have 
\[
\norm{\Mbold_{\Qrm}\grad\otimes\dots\otimes\Mbold_1\grad\otimes \ubold}_{\Mvarcal} 
= 
\TDVM{1,\ell+\Qrm-1}{}(\Mbold_{\Qrm-1}\grad\otimes \dots \otimes \Mbold_1\grad\otimes \ubold, \Mbold_\Qrm).
\]

\begin{definition}\label{def: total directional variation order Q}
The \emph{total directional variation of order $\Qrm$} of a $\Tcal^\ell(\RR^d)$-valued function $\ubold$ w.r.t.\ the collection of fields $\Mcal$ is defined as the Radon norm of $\Mbold_\Qrm\grad\otimes\dots\otimes\Mbold_1\grad\otimes \ubold$ and indicated as:
\begin{equation}
\TDVM{\Qrm,\ell}{\bm{1}}[\ubold,\Mcal] 
= 
\norm{\Mbold_\Qrm\grad\otimes\dots\otimes \Mbold_1\grad\otimes \ubold}_{\Mvarcal(\Omega,\Tcal^{\ell+\Qrm}(\RR^d))}.
\end{equation}
\end{definition}

\begin{definition}
Let $\Omega\subset\RR^d$ be a bounded Lipschitz domain and $\Mcal=(\Mbold_j)_{j=1}^\Qrm$ be a collection of smooth tensor fields such that $\Mbold_j\in\CCCspace{\infty}[\Omega,\Tcal^2(\RR^d)]$ for each $j=1,\dots,\Qrm$.
The space of $\Tcal^\ell(\RR^d)$-valued tensor functions $\ubold$ of \emph{bounded directional variation of order $\Qrm$} with respect to the collection of fields $\Mcal$ is defined as 
\begin{small}
\[
\BDVM{\Qrm}{}[\Omega,\Mcal,\Tcal^\ell(\RR^d)] 
= 
\left\{
\ubold\in\LL{1}[\Omega,\Tcal^\ell(\RR^d)]\,\Big\lvert\, \Mbold_\Qrm\grad\otimes\dots\otimes\Mbold_1\grad\otimes \ubold\in\Mvarcal(\Omega,\Tcal^{\ell+\Qrm}(\RR^d))
\right\}.
\]
\end{small}
\end{definition}

In particular, the spaces are nested and the larger is $\Qrm$, the smaller is the space. The space $\BDVM{\Qrm}{}[\Omega,\Mcal,\Tcal^\ell(\RR^d)]$ is endowed with the following norm:
\[
\norm{\ubold}_{\BDVM{\Qrm}{}}
= 
\norm{\ubold}_1 + \norm{\Mbold_\Qrm\grad\otimes\dots\otimes\Mbold_1\grad\otimes \ubold}_\Mvarcal.
\]

\begin{remark}
For fixed $\ell,\Qrm$ and by changing the weights $\alphabold$, $\TDVM{\Qrm,\ell}{\alphabold}$ yields equivalent norms and hence the same space. Thus, we can omit the weights in $\BDVM{\Qrm}{}[\Omega,\Mcal,\Tcal^\ell(\RR^d)]$.\end{remark}

\subsection{Properties}

\begin{proposition}\label{prop: 3.3 brehol}
Given $\Mcal=(\Mbold_j)_{j=1}^\Qrm$, $\TDVM{\Qrm,\ell}{\alphabold}(\blank,\Mcal)$ is a continuous semi-norm on $\BDVM{\Qrm}{}[\Omega,\Mcal,\Tcal^\ell(\RR^d)]$ with finite-dimensional kernel $\ker(\Mbold_\Qrm\grad\otimes\dots\otimes\Mbold_1\grad)$.
\end{proposition}
\begin{proof}
Positive homogeneity is ensured by definition of $\TDVM{\Qrm,\ell}{\alphabold}$: from the linearity of the integral we have
\[
\TDVM{\Qrm,\ell}{\alphabold}[\lambda \ubold,\Mcal] 
= 
\abs{\lambda}\TDVM{\Qrm,\ell}{\alphabold}[\ubold,\Mcal].
\]
For the triangular inequality, take $\ubold_1,\ubold_2\in\BDVM{\Qrm}{}[\Omega,\Mcal,\Tcal^\ell(\RR^d)]$ and let $\Psibold \in \Ycal_{\Mcal,\alphabold}^{\Qrm,\ell}$. Then:
\[
\TDVM{\Qrm,\ell}{\alphabold}[\ubold_1 + \ubold_2,\Mcal]
\leq 
\sup_{\Psibold}  \int_\Omega (\ubold_1 + \ubold_2)\cdot \div_{\Mcal}^\Qrm \left(\Psibold\right) \diff{\xbold}
\\
\leq
\TDVM{\Qrm,\ell}{\alphabold}[\ubold_1,\Mcal] + \TDVM{\Qrm,\ell}{\alphabold}[\ubold_2,\Mcal].
\]
For the continuity, let $\ubold_1,\ubold_2\in\BDVM{\Qrm}{}[\Omega,\Mcal,\Tcal^\ell(\RR^d)]$. Then it holds, exactly as in the $\BV{}{}$ case:
\[
\abs{\TDVM{\Qrm,\ell}{\alphabold}[\ubold_1,\Mcal]
-
\TDVM{\Qrm,\ell}{\alphabold}[\ubold_2,\Mcal]
}
\leq
\TDVM{\Qrm,\ell}{\alphabold}[\ubold_1-\ubold_2,\Mcal]
\leq
\norm{\ubold_1-\ubold_2}_{\BDVM{\Qrm}{}}.
\]
By definition of $\TDVM{\Qrm,\ell}{\alphabold}[\ubold_1,\Mcal]$, we have $\TDVM{\Qrm,\ell}{\alphabold}[\ubold_1,\Mcal]=0$ 
if and only if
\[
\int_\Omega \ubold \cdot \div_\Mcal^\Qrm \Psibold \diff \xbold = 0,
\quad 
\text{for each }
\Psibold \in\CCCcomp{\Qrm}[\Omega,\Tcal^{\ell+\Qrm}(\RR^d)],
\]
which is equivalent to $\ubold\in\ker(\Mbold_\Qrm\grad\otimes\dots\otimes\Mbold_1\grad)$ in the weak sense.
Therefore, $\TDVM{\Qrm,\ell}{\alphabold}$ is a semi-norm and $\BDVM{\Qrm}{}$ is a normed linear space.
From \cref{rem: 4.5}  $\ker(\Mbold_j\grad)$ on $\Dcal'(\Omega,\Tcal^{\ell+j}(\RR^d))$ has finite dimension for each $j=0,\dots,\Qrm-1$ then  $\ker(\Mbold_\Qrm\grad)$ considered on $\Dcal'(\Omega, \Tcal^{\ell+\Qrm}(\RR^d))$ is finite-dimensional and therefore $\ker(\Mbold_\Qrm\grad\otimes\dots\otimes\Mbold_1\grad)$ on $\Dcal'(\Omega, \Tcal^{\ell}(\RR^d))$ is finite-dimensional.
\end{proof}

\begin{proposition}\label{prop: convexity-lsc TDV}
$\TDVM{\Qrm,\ell}{\alphabold}(\blank,\Mcal)$ is convex and lower semi-continuous on $\BDVM{\Qrm}{}[\Omega,\Mcal,\Tcal^\ell(\RR^d)]$.
\end{proposition}
\begin{proof}
Fix $\Qrm,\ell\in\NN$, let $\Mcal$ be a collection of fields in $\Tcal^2(\RR^d)$ and let $\Psibold \in \Ycal_{\Mcal,\alphabold}^{\Qrm,\ell}$. 
Then for any $\Mcal$ and $\alphabold$ we take $\ubold_1,\, \ubold_2\in\LL{1}[\Omega,\Tcal^\ell(\RR^d)]$ and $t\in [0,1]$. Thus
\[
\begin{aligned}
\TDVM{\Qrm,\ell}{\alphabold}[t \ubold_1 + (1-t) \ubold_2,\Mcal] 
&= 
\sup_{\Psibold} \int_\Omega \left( t \ubold_1 + (1-t) \ubold_2 \right) \cdot \div_\Mcal^\Qrm \left(\Psibold \right) \diff \xbold
\\
&\leq 
t \sup_{\Psibold}  \int_\Omega \ubold_1 \cdot \div_\Mcal^\Qrm \left(\Psibold \right) \diff \xbold + 
(1-t)\sup_{\Psibold} \int_\Omega \ubold_2 \cdot \div_\Mcal^\Qrm\left(\Psibold \right) \diff \xbold\\
&= t \TDVM{\Qrm,\ell}{\alphabold}[\ubold_1,\Mcal] + (1-t)\TDVM{\Qrm,\ell}{\alphabold}[\ubold_2,\Mcal].
\end{aligned}
\] 
Hence $\TDVM{\Qrm,\ell}{\alphabold}$ is convex.
For the lower semi-continuity, let $(\ubold_j)_{j\in\NN}$ be a Cauchy sequence in $\BDVM{\Qrm}{}[\Omega,\Mcal,\Tcal^\ell(\RR^d)]$ such that $\ubold_j\to \ubold\in\LL{1}[\Omega,\Tcal^\ell(\RR^d)]$. 
From the definition of $\TDVM{\Qrm,\ell}{\alphabold}$, we have:
\[
\int_\Omega 
\ubold\cdot\div_\Mcal^\Qrm \Psibold
\diff\xbold 
= 
\lim_{j\to\infty}
\int_\Omega 
\ubold_j \cdot \div_\Mcal^\Qrm\Psibold
\diff\xbold 
\\
\leq  
\liminf_{j\to\infty} 
\TDVM{\Qrm,\ell}{\alphabold}[\ubold_j,\Mcal],
\quad
\text{for any}\quad\Psibold\in\Ycal_{\Mcal,\alphabold}^{\Qrm,\ell}.
\]
Then, taking the supremum we have
$
\TDVM{\Qrm,\ell}{\alphabold}[\ubold,\Mcal] 
\leq  
\liminf_{j\to\infty} \TDVM{\Qrm,\ell}{\alphabold}[\ubold_j,\Mcal].
$
\end{proof}

Similarly to \cite[Proposition 3.5]{BreKunPoc2010}, the space $\BDVM{\Qrm}{}[\Omega,\Mcal,\Tcal^\ell(\RR^d)]$ is a Banach space when equipped with a suitable norm:
\begin{proposition}$\BDVM{\Qrm}{}[\Omega,\Mcal,\Tcal^\ell(\RR^d)]$ endowed with the norm
\[
\norm{\ubold}_{\BDVM{\Qrm}{}[\Omega,\Mcal,\Tcal^\ell(\RR^d)]} 
= 
\norm{\ubold}_{\LL{1}[\Omega,\Tcal^\ell(\RR^d)]}  
+ 
\TDVM{\Qrm,\ell}{\alphabold}[\ubold,\Mcal].
\]  
is a Banach space.
\end{proposition}
\begin{proof}
We have already proved in \cref{prop: convexity-lsc TDV} the lower semi-continuity of $\TDVM{\Qrm,\ell}{\alphabold}$. 
As in \cite{BreKunPoc2010}, let $(\ubold_j)_{j\in\NN}$ be a Cauchy sequence in $\BDVM{\Qrm}{}[\Omega,\Mcal,\Tcal^\ell(\RR^d)]$. Then it is easy to see that $(\ubold_j)$ is a Cauchy sequence in $\LL{1}[\Omega,\Tcal^\ell(\RR^d)]$ and a limit $\ubold\in \LL{1}[\Omega,\Tcal^\ell(\RR^d)]$ exists. Thus, by lower semi-continuity we have:
\[
\TDVM{\Qrm,\ell}{\alphabold}[\ubold,\Mcal] 
\leq  
\liminf_{j\to\infty} \TDVM{\Qrm,\ell}{\alphabold}[\ubold_j,\Mcal].
\]
So, $\ubold\in\BDVM{\Qrm}{}[\Omega,\Mcal,\Tcal^\ell(\RR^d)]$ and we need only to check that $\ubold$ is the limit in the corresponding norm: being $(\ubold_j)_{j\in\NN}$ a Cauchy sequence, then we can choose $\varepsilon>0$ and an index $j^\ast$ such that for all $j>j^\ast$ we have
\[
\TDVM{\Qrm,\ell}{\alphabold}[\ubold_{j^\ast}-\ubold_{j},\Mcal]
\leq 
\varepsilon.
\]
Letting $j\to \infty$, the lower semi-continuity of $\TDVM{\Qrm,\ell}{\alphabold}$ on $\LL{1}[\Omega,\Tcal^\ell(\RR^d)]$ gives
\[
\TDVM{\Qrm,\ell}{\alphabold}[\ubold_{j^\ast}-\ubold,\Mcal]
\leq 
\liminf_{j\to\infty} \TDVM{\Qrm,\ell}{\alphabold}[\ubold_{j^\ast}-\ubold_j,\Mcal]\leq \varepsilon,
\]
and this implies that $\ubold_j\to \ubold$ in $\BDVM{\Qrm}{}[\Omega,\Mcal,\Tcal^\ell(\RR^d)]$.
\end{proof}

\section{Equivalent representation}\label{sec: equivalent}
We are going to interpret the dual definition of the regulariser $\TDVM{\Qrm,\ell}{\alphabold}[\ubold,\Mcal]$ in terms of iterated Fenchel duality
following the proof given in \cite{BreHol2014}.
Firstly, we prove the following preliminary result similarly to \cite[Lemma 3.4]{BreHol2014}.
\begin{lemma}\label{lem: 3.4}
Let $j\geq 1$ and let $\zbold_{j-1}\in \CCospace{j-1}[\Omega,\Tcal^{\ell+j-1}(\RR^d)]^\ast$, $\zbold_{j}\in \CCospace{j}[\Omega,\Tcal^{\ell+j}(\RR^d)]^\ast$ be distributions of order $j-1$ and $j$, respectively. 
Then
\begin{equation}
\norm{\Mbold_j\grad \zbold_{j-1}-\zbold_j}_{\Mvarcal} 
= 
\sup_{
\substack{\Psibold\in\CCCcomp{\infty}[\Omega,\Tcal^{\ell+j}(\RR^d)],\\ \norm{\Psibold}_\infty\leq 1}} 
\left\{ 
\langle \zbold_{j-1},\, \div_{\Mbold_j}\Psibold \rangle 
+
\langle \zbold_j,\,\Psibold \rangle
\right\},
\label{eq: sup lem 3.4}
\end{equation}
with the right-hand side being finite if and only if $\Mbold_j\grad\zbold_{j-1}-\zbold_j \in \Mvarcal\left(\Omega,\Tcal^{\ell+j}(\RR^d)\right)$ in the distributional sense.
\end{lemma}
\begin{proof}
In the distributional sense, we have for all $\Psibold\in\CCCcomp{\infty}[\Omega,\Tcal^{\ell+j}(\RR^d)]$:
\[
\langle 
\zbold_j - \Mbold_j\grad\zbold_{j-1},\,\Psibold
\rangle 
= 
\langle \zbold_j,\,\Psibold \rangle
+
\langle 
\zbold_{j-1},\,\div_{\Mbold_j}\Psibold
\rangle.
\]
Since $\CCCcomp{\infty}[\Omega,\Tcal^{\ell+j}(\RR^d)]$ is dense in $\CCospace{}[\Omega,\Tcal^{\ell+j}(\RR^d)]$, the distribution $\zbold_j-\Mbold_j\grad\zbold_{j-1}$ can be extended to an element in $\CCospace{}[\Omega,\Tcal^{\ell+j}(\RR^d)]^\ast=\Mvarcal\left(\Omega,\Tcal^{\ell+j}(\RR^d)\right)$ if and only if the supremum in \cref{eq: sup lem 3.4} is finite, in which case it coincides with the Radon norm by definition.
\end{proof}

Finally, we are now ready to show the \emph{minimum representation} of $\TDVM{\Qrm,\ell}{\alphabold}$, similarly to \cite[Theorem 3.5]{BreHol2014}.

\begin{proposition}\label{prop: equivalence TDV}
Let $\ubold\in\LLloc{1}[\Omega,\Tcal^\ell(\RR^d)]$, $\TDVM{\Qrm,\ell}{\alphabold}[\ubold,\Mcal]$ be defined as in \cref{def: TDV} and $\Mcal=(\Mbold_j)_{j=1}^\Qrm$ be a collection of positive definite tensor fields such that $\Mbold_j\in\Tcal^2(\RR^d)$ for all $j$. 
Then it holds  
\begin{equation}
\TDVM{\Qrm,\ell}{\alphabold}[\ubold,\Mcal] 
= 
\min_{
\substack{\zbold_j\in\BDVM{}{}[\Omega,\Mbold_{j+1},\Tcal^{\ell+j}(\RR^d)]
\\ 
j=1,\dots,\Qrm-1, 
\\
\zbold_0=\ubold,\, \zbold_\Qrm=\bm{0}}
}
\sum_{j=1}^\Qrm 
\alpha_{\Qrm-j}\norm{\Mbold_{j}\grad \zbold_{j-1} - \zbold_j}_{\Mvarcal},
\label{eq: equivalence TDV}
\end{equation}
with the minimum being finite if and only if 
$\zbold_j\in \BDVM{}{}[\Omega,\Mbold_{j+1},\Tcal^{\ell+j}(\RR^d)]$ for each $j = 0,\dots,\Qrm-1$, with $\zbold_0 = \ubold$ and $\zbold_\Qrm = \bm{0}$.
\end{proposition}
\begin{proof}
Let $\ubold\in\LLloc{1}[\Omega,\Tcal^\ell(\RR^d)]$ be such that $\TDVM{\Qrm,\ell}{\alphabold}[\ubold,\Mcal]<\infty$. 
In order to make use of the Fenchel-Rockafellar duality we introduce the following Banach spaces:
\[
\begin{aligned}
X 
&= 
\CCospace{1}[\Omega,\Tcal^{\ell+1}(\RR^d)]\times\dots\times\CCospace{\Qrm}[\Omega,\Tcal^{\ell+\Qrm}(\RR^d)],
\\
Y 
&= 
\CCospace{1}[\Omega,\Tcal^{\ell+1}(\RR^d)]\times\dots\times\CCospace{\Qrm-1}[\Omega,\Tcal^{\ell+\Qrm-1}(\RR^d)].
\end{aligned}
\]
Let $\zbold=(\zbold_1,\dots,\zbold_{\Qrm-1})\in Y$ be the primal variable,  $\wbold=(\wbold_1,\dots,\wbold_\Qrm)\in X$ be the dual variables and $\Kcal\in\Lcal(X,Y)$ be the linear operator defined as
\[
\Kcal =
\begin{pmatrix}
 -\Ibold & -\div_{\Mbold_2} & \bm{0} &\dots & \dots &\dots & \dots & \bm{0} \\ 
\bm{0} &  -\Ibold & -\div_{\Mbold_3} & \bm{0} & & & & \vdots \\ 
\vdots & \ddots & \ddots & \ddots & \ddots & & & \vdots \\ 
\vdots & &  \bm{0} & -\Ibold & -\div_{\Mbold_{j+1}} & \bm{0} & & \vdots \\ 
\vdots &  &  &  \ddots & \ddots & \ddots & \ddots & \vdots \\
\vdots &  &  &  & \ddots & \ddots & \ddots & \bm{0} \\  
\bm{0} & \ddots & \dots & \dots & \dots  & \bm{0} & -\Ibold & -\div_{\Mbold_{\Qrm}} \\ 
\end{pmatrix},
\]
such that
\[
\Kcal
\begin{pmatrix}
\wbold_1\\
\vdots\\
\wbold_\Qrm
\end{pmatrix}
=
\begin{pmatrix}
-\wbold_1 - \div_{\Mbold_{2}} \wbold_2\\
\vdots \\
-\wbold_{j} - \div_{\Mbold_{j+1}} \wbold_{j+1}\\
\vdots \\
-\wbold_{\Qrm-1} - \div_{\Mbold_{\Qrm}} \wbold_{\Qrm}
\end{pmatrix}.
\]
Let the proper, convex and lower semi-continuous functionals
\[
\begin{aligned}
F &: X\to ]-\infty,\infty],
& 
F(\wbold) 
&= 
-\langle \ubold,\, \div_{\Mbold_1} \wbold_1 \rangle 
+ 
\sum_{j=1}^{\Qrm} 
I_{\{\norm{\blank}_\infty\leq \alpha_{\Qrm-j}\}}(\wbold_j)
\\
G &: Y\to ]-\infty,\infty], 
& G(\zbold) &= I_{\bm{0}}(\zbold),
\end{aligned}
\]
where $I_Z$ is the indicator function of this set, i.e.\ $I_Z(\zbold) = 0$ if $\zbold\in Z$ and $I_Z(\zbold) =\infty$ otherwise. 
Then, the following identity holds from \cref{def: TDV}:
\[
\TDVM{\Qrm,\ell}{\alphabold}[\ubold,\Mcal] 
= 
\sup_{\wbold\in X} -F(\wbold) - G(\Kcal\wbold).
\]

In the next, we want to obtain the following result:
\begin{equation}
\TDVM{\Qrm,\ell}{\alphabold}[\ubold,\Mcal] 
= 
\min_{\wbold^\ast\in Y^\ast} F^\ast(-\Kcal^\ast \wbold^\ast) + G^\ast(\wbold^\ast).
\label{eq: minimum representation}
\end{equation}
This follows from \cite[Corollary 2.3]{Attouch1986}, once we show
\[
Y = \bigcup_{\lambda>0} \lambda\left(\dom(G)-\Kcal \dom(F)\right).
\]
Indeed, let $\zbold\in Y$ and define recursively:
\[
\begin{aligned}
\wbold_\Qrm 
&=\bm{0} 
&&\in\CCospace{\Qrm}[\Omega,\Tcal^{\ell+\Qrm}(\RR^d)],
\\
\wbold_{\Qrm-1} 
&=\zbold_{\Qrm-1} - \div_{\Mbold_{\Qrm}}\wbold_{\Qrm} 
&&\in\CCospace{\Qrm-1}[\Omega,\Tcal^{\ell+\Qrm-1}(\RR^d)],
\\
&\vdotswithin{=}
\\
\wbold_j 
&= \zbold_j - \div_{\Mbold_{j+1}}\wbold_{j+1} 
&&\in \CCospace{j}[\Omega,\Tcal^{\ell+j}(\RR^d)],  
\\
&\vdotswithin{=}
\\
\wbold_1 
&= \zbold_1 - \div_{\Mbold_{2}}\wbold_{2} 
&&\in \CCospace{1}[\Omega,\Tcal^{\ell+1}(\RR^d)].
\end{aligned}
\]
Hence, $\wbold\in X$ and $-\Kcal\wbold = \zbold\in Y$. 
Moreover, for $\lambda>0$ large enough, we have
\[
\norm{\lambda^{-1}\wbold}_{\infty} \leq \alpha_{\Qrm-j},\quad\text{for all }j=1,\dots,\Qrm.
\]
Therefore, from $\lambda^{-1}\wbold\in\dom(F)$ and $\bm{0}\in\dom(G)$, we get the following representation:
\[
\zbold = \lambda(\bm{0}-\Kcal\lambda^{-1}\wbold).
\]
This means that \cref{eq: minimum representation} holds and the minimum is obtained in $Y^\ast$, which can be written as
\[
Y^\ast = \left(\CCospace{1}[\Omega,\Tcal^{\ell+1}(\RR^d)]\right)^\ast \times \dots \times \left(\CCospace{1}[\Omega,\Tcal^{\ell+\Qrm-1}(\RR^d)]\right)^\ast
\]
and $\zbold^\ast=(\zbold_1^\ast,\dots,\zbold_{\Qrm-1}^\ast)$, $\zbold_{j}^\ast\in\CCospace{j}[\Omega,\Tcal^{\ell+j}(\RR^d)]$, for $1\leq j\leq \Qrm-1$. 
Hence, imposing $\zbold_0^\ast = \ubold$ and $\zbold_{\Qrm}^\ast = \bm{0}$, from $G^\ast=\bm{0}$ the following chain holds:
\[
\begin{aligned}
F^\ast(-\Kcal^\ast\zbold^\ast) + G^\ast(\zbold^\ast)& 
= 
\sup_{\wbold\in X}  
\left( 
\langle 
-\Kcal^\ast\zbold^\ast,\,\wbold
\rangle 
+ 
\langle 
\ubold,\,\div_{\Mbold_1}\wbold_1
\rangle 
-\sum_{j=1}^\Qrm 
I_{\{\norm{\blank}_{\infty}\leq\alpha_{\Qrm-j}\}}(\wbold_j)
\right)
\\
&= 
\sup_{
\substack{
\wbold\in X \\ \norm{\wbold_j}_\infty \leq \alpha_{\Qrm-j} 
\\ 
j=1,\dots,\Qrm}
}
\left( 
\langle \ubold ,\, \div_{\Mbold_1}\wbold_1 \rangle 
+ 
\sum_{j=1}^{\Qrm-1} 
\langle 
\zbold_j^\ast,\, \div_{\Mbold_{j+1}} \wbold_{j+1} 
+\wbold_j
\rangle
\right)
\\
&= 
\sum_{j=1}^\Qrm \alpha_{\Qrm-j} 
\left( 
\sup_{
\substack{
\wbold_j\in \CCospace{j}[\Omega,\Tcal^{\ell+j}(\RR^d)],
\\ 
\norm{\wbold_j}_\infty\leq 1}
}
\langle \zbold_{j-1}^\ast,\,\div_{\Mbold_j}\wbold_j \rangle  
+ \langle \zbold_{j}^\ast,\,\wbold_j \rangle  
\right).
\end{aligned}
\]
From \cref{lem: 3.4} we have that each supremum is finite and
\[
\sup_{
\substack{
\wbold_j\in \CCCcomp{j}[\Omega,\Tcal^{\ell+j}(\RR^d)],
\\ 
\norm{\wbold_j}_\infty\leq 1}
}
\langle \zbold_{j-1}^\ast,\,\div_{\Mbold_j}\wbold_j \rangle  
+ \langle \zbold_{j}^\ast,\,\wbold_j \rangle   
= 
\norm{\Mbold_j\grad \zbold_{j-1}^\ast-\zbold_j^\ast}_{\Mvarcal} 
\]
if and only if $\Mbold_j\grad \zbold_{j-1}^\ast-\zbold_j^\ast \in \Mvarcal\left(\Omega,\Tcal^{\ell+j}(\RR^{d})\right)$, for $j=1,\dots,\Qrm$. 
Since, $\zbold_\Qrm^\ast=\bm{0}$, by \cref{th: 2.6} this means that $\zbold_{\Qrm-1}^\ast\in\BDVM{}{}[\Omega,\Mbold_{\Qrm},\Tcal^{\ell+\Qrm-1}(\RR^d)]$, so
\[
\zbold_{\Qrm-1}^\ast \in \Mvarcal\left(\Omega,\Tcal^{\ell+\Qrm-1}(\RR^{d})\right).
\]
By induction, we have $\zbold_j^\ast\in\BDVM{}{}[\Omega,\Mbold_{j+1},\Tcal^{\ell+j}(\RR^d)]$ for each $j=0,\dots,\Qrm$ so we can take the minimum in \cref{eq: minimum representation} over all BDV-tensor fields, obtaining \cref{eq: equivalence TDV}: such minimum is finite if $\ubold\in\BDVM{}{}[\Omega,\Mcal,\Tcal^\ell(\RR^d)]$. 
\end{proof}

\begin{remark}\label{rem: minimal representation remark}
Let $\alphabold_{j} = (\alpha_0,\dots,\alpha_{j})$ be such that  $\alphabold_j \subseteq\alphabold=(\alpha_0,\dots,\alpha_{\Qrm-1})$ and
let $\Mcal_{\Qrm-j}^-=(\Mbold_{\Qrm-j+1},\dots,\Mbold_\Qrm)$ be a subset of $\Mcal=(\Mbold_1,\dots,\Mbold_\Qrm)$ such that $\Mcal_{\Qrm-j}^-\subseteq\Mcal$. Then the regulariser $\TDVM{\Qrm,\ell}{\alphabold}[\ubold,\Mcal]$ can be expressed recursively as:
\begin{small}
\[
\begin{aligned}
\TDVM{1,\ell+\Qrm-1}{\alpha_0}[\zbold_{\Qrm-1},\Mcal_{\Qrm-1}^-] 
&= 
\alpha_0 \norm{\Mbold_\Qrm \grad \otimes\zbold_{\Qrm-1}}_\Mvarcal,
\\
&\vdotswithin{=}
\\
\TDVM{j+1,\ell+\Qrm-j-1}{\alphabold_j}[\zbold_{\Qrm-j},\Mcal_{\Qrm-j-1}^-] 
&= 
\begin{aligned}
\min_{\zbold_{\Qrm-j}}
&
\Big(\alpha_{j} \norm{\Mbold_{\Qrm-j}\grad \otimes \zbold_{\Qrm-j-1} - \zbold_{\Qrm-j}}_\Mvarcal 
\\ &+
\TDVM{j,\ell+\Qrm-j}{\alphabold_{j-1}}[\zbold_{\Qrm-j}, \Mcal_{\Qrm-j}^-]\Big),
\end{aligned}
\\
&\vdotswithin{=}
\\
\TDVM{\Qrm,\ell}{\alphabold}[\ubold,\Mcal_0^-]
&= 
\min_{\zbold_1}
\left(
\alpha_{\Qrm-1}\norm{\Mbold_1\grad\otimes \ubold - \zbold_1}_\Mvarcal
+ 
\TDVM{\Qrm-1,\ell+1}{\alphabold_{\Qrm-2}}(\zbold_1,\Mcal_{1}^-)
\right),
\end{aligned}
\]
\end{small}
where $\zbold_j\in\BDVM{}{}[\Omega,\Mbold_{j+1},\Tcal^{\ell+j}(\RR^d)]$.
\end{remark}

\begin{remark}
As in \cite[Remark 3.8]{BreHol2014}, the minimum representation is monotonic with respect to the weights. Indeed let $\alphabold,\betabold\in\RR_+^\Qrm$ with $\alpha_j\leq\beta_j$ for each $j=0,\dots,\Qrm-1$. Then
\[
\TDVM{\Qrm,\ell}{\alphabold}[\ubold,\Mcal]
\leq 
\TDVM{\Qrm,\ell}{\betabold}[\ubold,\Mcal].
\]
\end{remark} 
\section{Existence of TDV-regularised solutions}\label{sec: existence}
In this section 
we prove the existence of solutions to TDV-regularised problems of the type:
\begin{equation}
\min_{\ubold\in\LL{p}[\Omega,\Tcal^\ell(\RR^d)]} 
\TDVM{\Qrm,\ell}{\alphabold}[\ubold,\Mcal] + F(\ubold),
\label{eq: minimization problem}
\end{equation}
where $F:\LL{p}[\Omega,\Tcal^\ell(\RR^d)]\to \RR$ is a fidelity term.
In the next, we will follow \cite{BreHol2014} so as to check that the same results hold in our weighted case and we will proceed often by induction on $\Qrm$. 
We proceed by proving the embedding theorems and the existence of a minimiser for \cref{eq: minimization problem}.

\subsection{Embeddings}
We state some results in view of the embedding \cref{th: bredies 4.16,th: 4.17 bredies}. 
The following Sobolev-Korn type inequality holds for smooth tensor fields with compact support, similarly to \cite[Theorem 4.8]{Bredies2012}.
\begin{lemma}\label{lem: poincare}
Let $\ubold\in\CCCcomp{1}[\Omega,\Tcal^\ell(\RR^d)]$ and $\Mbold\in\LL{\infty}(\Omega,\Tcal^2(\RR^d))$ be a field of invertible matrices for every $\xbold\in\Omega$ such
that $S=\sup_{\xbold}\norm{(\Mbold(\xbold))^{-1}}_2<\infty$. 
Then there exists a constant $C$ depending only on $\Omega,\ell$ and $S$ such that
\[
\norm{\ubold}_{d/(d-1)}
\leq C 
\norm{\Mbold \grad \otimes \ubold}_1.
\]
\end{lemma}
\begin{proof}
Let $\norm{\blank}$ be the operator norm. We have the desired inequality, where the first one is due to the standard Sobolev inequality for tensor-valued functions:
\[
\norm{\ubold}_{d/(d-1)} 
\leq C_1 
\norm{\grad\otimes\ubold}_{1} = 
C_1 \norm{(\Mbold(\cdot))^{-1}\Mbold(\cdot) \grad\otimes\ubold}_1
\leq
C_1 S \norm{\Mbold\grad\otimes \ubold}_1,
\]
and the conclusion follows with $C:=C_1S$.
\end{proof}

The following lemma states a result similar to \cite[Lemma 3.9]{BreHol2014}.
\begin{lemma}\label{lem 3.9 brehol}
For each $\Qrm\geq 1$, $\ell\geq 0$ there exists a constant $C_1>0$ depending only on $\Omega,\,\Qrm$ and $\ell$ such that for each $\ubold\in\BDVM{}{}[\Omega,\Mbold,\Tcal^\ell(\RR^d)]$ and $\overline{\wbold}\in\ker(\TDVM{\Qrm,\ell+1}{\alphabold})\subset\LL{1}[\Omega,\Tcal^\ell(\RR^d)]$:
\[
\norm{\Mbold\grad\otimes \ubold}_\Mvarcal \leq C_1 ( \norm{\ubold}_1 + \norm{\Mbold\grad\otimes \ubold - \overline{\wbold}}_\Mvarcal).
\]
\end{lemma}
\begin{proof}
We argue by contradiction. Suppose that there exists $\Qrm$ and $\ell$ such that the bound does not hold. 
Then there exist $(\ubold_j)_{j\in\NN}$ and $(\overline{\wbold}_j)_{j\in\NN}$, with each $\ubold_j\in\BDVM{}{}[\Omega,\Mbold,\Tcal^\ell(\RR^d)]$ and $\overline{\wbold}_j\in\ker(\TDVM{\Qrm,\ell+1}{\alphabold})$ such that
\[
\norm{\Mbold\grad\otimes \ubold_j}_\Mvarcal 
= 1
\quad\text{and}\quad 
\norm{\ubold_j}_1 + \norm{\Mbold\grad\otimes \ubold_j -\overline{\wbold}_j}_\Mvarcal\leq j^{-1}.
\]
Thus $(\overline{\wbold}_j)_{j\in\NN}$ is bounded with respect to the norm $\norm{\blank}_\Mvarcal$ in the finite dimensional space $\ker(\TDVM{\Qrm,\ell+1}{\alphabold})$.
Therefore, there exists a subsequence relabelled as $(\overline{\wbold}_j)_{j\in\NN}$ and converging to $\overline{\wbold}\in\ker(\TDVM{\Qrm,\ell+1}{\alphabold})$ in the $\LL{1}[\Omega, \Tcal^{\ell+1}(\RR^d)]$ norm and thus, $\Mbold\grad\otimes \ubold_j\to \overline{\wbold}$.
Moreover, $\ubold_j\to \bm{0}$ in $\LL{1}[\Omega, \Tcal^\ell(\RR^d)]$ implies that  $\Mbold\grad\otimes \ubold_j\to \bm{0}$ in $\Mvarcal$ by closedness of the gradient and this contradicts $\norm{\Mbold\grad\otimes \ubold}_\Mvarcal=1$.
\end{proof}

We can also define the zero extension $E\ubold$ of a function $\ubold$ of bounded directional variation. Such zero extension has bounded directional variation as can be proved adapting \cite[Corollary 4.15]{Bredies2012} based on \cite[Theorem 4.12]{Bredies2012}.
\begin{corollary}\label{corBred}
Let $\Omega$ a bounded Lipschitz domain and $\ubold\in\BDVM{}{}[\Omega,\Mbold,\Tcal^\ell(\RR^d)]$. Then the zero extension $E\ubold$ is in $\BDVM{}{}[\RR^d,\Mbold,\Tcal^\ell(\RR^d)]$. In addition, there exists $C>0$ such that for all $\ubold\in\BDVM{}{}[\Omega,\Mbold,\Tcal^\ell(\RR^d)]$:
\begin{equation}
\norm{E\ubold}_1+\norm{\Mbold\grad\otimes(E\ubold)}_\Mvarcal\leq C \left(\norm{\ubold}_1+\norm{\Mbold\grad\otimes\ubold}_\Mvarcal\right).
\label{eq: zero extension continuity}
\end{equation}
\end{corollary}
\begin{proof}
It follows by adapting the proof of \cite[Corollary 4.15, Theorem 4.12]{Bredies2012}.
\end{proof}

In the next theorem, we prove the continuous embedding of  the space $\BDVM{}{}$ into $\LL{d/(d-1)}$, similarly to \cite[Theorem 4.16]{Bredies2012}.
\begin{theorem}\label{th: bredies 4.16}
Let $\Omega\subset\RR^d$ be a bounded Lipschitz domain. Then, there is a continuous injection
\[
\BDVM{}{}[\Omega,\Mbold,\Tcal^\ell(\RR^d)] 
\xhookrightarrow{} 
\LL{d/(d-1)}[\Omega,\Tcal^\ell(\RR^d)].
\]
\end{theorem}
\begin{proof}
In this proof we follow \cite[Theorem 4.16]{Bredies2012}, with the notational changes $\Mbold\grad$, $\Tcal^\ell(\RR^d)$ and $\BDVM{}{}[\Omega,\Mbold,\Tcal^\ell(\RR^d)]$ in place of the symmetrised gradient $\Ecal$, $\Sym{\ell}(\RR^d)$ and $\BD{}[\Omega,\Sym{\ell}[\RR^d]]$, respectively, and $d\geq 2$.
If $\ubold\in\BDVM{}{}[\Omega,\Mbold,\Tcal^\ell(\RR^d)]\cap \CCCcomp{1}[\Omega,\Tcal^\ell(\RR^d)]$ then \cref{lem: poincare} gives the result.
In the general case $\ubold\in\BDVM{}{}[\Omega,\Mbold,\Tcal^\ell(\RR^d)]$, its zero extension $E\ubold$ can be approximated by a sequence of strictly converging continuously differentiable, compactly supported functions $(\ubold_j)_{j\in\NN}$, by applying \cref{prop: 4.13 Bredies2012} to a bounded domain $\Omega'$ such that $\overline{\Omega}\subset\subset\Omega'$.
According to the estimate in 
\cref{lem: poincare} we have for each $j$
\[
\norm{\ubold_j}_{d/(d-1)} 
\leq 
C\norm{\Mbold \grad \otimes \ubold_j}_1
=
C\norm{\Mbold \grad \otimes \ubold_j}_\Mvarcal
\leq
C\left( \norm{\ubold_j}_1 + \norm{\Mbold\grad\otimes \ubold_j}_\Mvarcal\right).
\]
Now $\ubold_j\to E\ubold$ in $\LL{1}[\Omega, \Tcal^\ell(\RR^d)]$, and by the lower semicontinuity of the $\LL{d/(d-1)}[\Omega, \Tcal^\ell(\RR^d)]$-norm, the strict convergence in $\BDVM{}{}$, and Corollary~\ref{corBred} we get:
\begin{equation}
\begin{aligned}
\norm{\ubold}_{d/(d-1)} = 
\norm{E\ubold}_{d/(d-1)}
&\leq
C(\norm{\ubold}_1 + \norm{\Mbold\grad \otimes \ubold}_\Mvarcal).
\end{aligned}
\label{th: bredies 4.16 estimate}
\end{equation}
\end{proof}

Now, we show that the embedding in \cref{th: bredies 4.16} is compact for $1 \leq p < d/(d-1)$, similarly to \cite[Theorem 4.17]{Bredies2012}.
\begin{theorem}\label{th: 4.17 bredies}
Let $\Omega$ be a bounded Lipschitz domain, $1\leq p < d/(d-1)$ and $(\ubold_j)_{j\in\NN}$ be a bounded sequence in $\BDVM{}{}[\Omega,\Mbold,\Tcal^\ell(\RR^d)]$. 
Then, a subsequence $(\ubold_{j_\ell})_{\ell\in\NN}$ converges in $\LL{p}[\Omega,\Tcal^\ell(\RR^d)]$.\end{theorem}
\begin{proof}
We aim to prove the compact embedding $\BDVM{}{}[\Omega,\Mbold,\Tcal^\ell(\RR^d)]\hookrightarrow\LL{1}[\Omega,\Tcal^\ell(\RR^d)]$, i.e.\ by fixing $\Omega'$ such that $\overline{\Omega}\subset\subset\Omega'$ and $\ubold\in\CCCcomp{2}[\RR^d,\Tcal^\ell(\RR^d)]$ with support in $\Omega'$, then
\[
\int_{\RR^d}\abs{\ubold(\overline{\xbold}+\hbold) - \ubold(\overline{\xbold})}\diff{\overline{\xbold}}
\leq 
C
\abs{\hbold}^s\norm{\Mbold\grad\otimes \ubold}_1
\]
for some $s>0$ and all $\hbold\in\RR^d$, $\abs{\hbold}\leq 1$ with a constant $C$ independent of $\ubold$. This part follows by the same argument as in the first part of the proof of \cite[Theorem 4.17]{Bredies2012}.

Let $\ubold\in\BDVM{}{}[\Omega,\Mbold,\Tcal]$ be arbitrary. 
The zero extension $E\ubold\in\BDVM{}{}[\Omega',\Mbold,\Tcal^\ell(\RR^d)]$ has compact support in $\Omega'$ and thus there exists a smooth sequence $(\ubold_j)_{j\in\NN}$ in $\CCCcomp{\infty}[\Omega,\Tcal^\ell(\RR^d)]$ such that $\ubold_j\to E \ubold$ in $\LL{1}[\Omega,\Tcal^\ell(\RR^d)]$ and $\norm{\Mbold\grad\otimes \ubold_j}_1\to \norm{\Mbold\grad\otimes E\ubold}_\Mvarcal$ as $j\to\infty$. Thus:
\[
\int_{\RR^d} \abs{E\ubold(\overline{\xbold}+\hbold) - E \ubold(\overline{\xbold})}\diff \overline{\xbold} \leq C \abs{\hbold}^s \norm{\Mbold\grad\otimes E\ubold}_\Mvarcal \leq C \abs{\hbold}^s \left(\norm{\ubold}_1 + \norm{\Mbold\grad\otimes \ubold}_\Mvarcal\right).
\]
For a bounded sequence in $\BDVM{}{}[\Omega,\Mbold,\Tcal^\ell(\RR^d)]$ we have $(E\ubold_j)_{j\in\NN}$ relatively compact, thus there exist $\ubold\in\LL{1}[\RR^d,\Tcal^\ell(\RR^d)]$ and a subsequence $(E\ubold_{j_\ell})_{\ell\in\NN}$ with $E \ubold_{j_\ell}\to \ubold$.
Also, $\ubold_{j_\ell}\to \ubold_{\at\Omega}$ in $\LL{1}[\Omega,\Tcal^\ell(\RR^d)]$ proving the compact embedding $\BDVM{}{}[\Omega,\Mbold,\Tcal^\ell(\RR^d)]$ in $\LL{1}[\Omega,\Tcal^\ell(\RR^d)]$.

For the general case $1\leq p < d/(d-1)$, it follows from Theorem~\ref{th: bredies 4.16} that $(\ubold_j)_{j\in\NN}$ is bounded in $\LL{d/(d-1)}[\Omega,\Tcal^\ell(\RR^d)]$, and the result follows from an application of Young's inequality as in the proof of \cite[Theorem 4.17]{Bredies2012}.
\end{proof}

Every bounded sequence in $\BDVM{}{}[\Omega,\Mbold,\Tcal^\ell(\RR^d)]$ admits a subsequence which converges in the weak-$\ast$ sense, while strict convergence implies  weak-$\ast$ convergence. The embeddings above allow to reinterpret weak-$\ast$ sequences in $\BDVM{}{}[\Omega,\Mbold,\Tcal^\ell(\RR^d)]$ as:
\begin{itemize}
\item  weakly converging sequences in $\LL{d/(d-1)}[\Omega,\Tcal^\ell(\RR^d)]$ (weak-$\ast$ for $d=1$);
\item strongly converging sequences in $\LL{p}[\Omega,\Tcal^\ell(\RR^d)]$ for any $p\in[1,d/(d-1)[$, continuously.
\end{itemize}
Also, $\CCCspace{\infty}[\overline{\Omega},\Tcal^\ell(\RR^d)]$ is dense in $\BDVM{}{}[\Omega,\Mbold,\Tcal^\ell(\RR^d)]$, with respect to strict convergence.

\subsection{Existence}
In what follows, we prove the coercivity for $\TDVM{\Qrm,\ell}{\alphabold}$ in view of satisfying the conditions of the Tonelli-Weierstra\ss\, theorem for the minimisation problem \cref{eq: minimization problem}.
\begin{definition}\label{def: R}
For each $\Qrm\geq 1$ and $\ell\geq 0$ let $R_{\Qrm,\ell}$ be a linear, continuous and onto projection such that
\[
R_{\Qrm,\ell} 
: 
\LL{d/(d-1)}[\Omega,\Tcal^\ell(\RR^d)]\to\ker(\Mbold_\Qrm\grad\otimes\dots\Mbold_1\grad).
\]
\end{definition}

Note that $R_{\Qrm,\ell}$ defined as above always exists since $\ker(\TDVM{\Qrm,\ell}{\alphabold})=\ker(\Mbold_\Qrm\grad\otimes\dots\Mbold_1\grad)$ is finite dimensional.

The following coercivity estimate holds, similarly to \cite[Proposition 3.11]{BreHol2014}.

\begin{proposition}\label{prop: coercivity}
For each $\Qrm\geq 1$ and $\ell\geq 0$, there exists a constant $C>0$ such that for all $\ubold\in\LL{d/(d-1)}[\Omega,\Tcal^\ell(\RR^d)]$:
\begin{small}
\[
\norm{\Mbold_1\grad \otimes \ubold}_\Mvarcal 
\leq C
\left(\norm{\ubold}_1 + \TDVM{\Qrm,\ell}{\alphabold}[\ubold,\Mcal]\right)
\quad \text{and} 
\quad
\norm{\ubold- R_{\Qrm,\ell} \ubold}_{d/(d-1)} \leq C\, \TDVM{\Qrm,\ell}{\alphabold}[\ubold,\Mcal].
\]
\end{small}
\end{proposition}
\begin{proof}
Exactly as in the proof of \cite[Proposition 3.11]{BreHol2014}, we proceed by induction on $\Qrm$. Let $\Qrm=1$ and $\ell\geq 0$.
Then the first inequality is trivial while the second follows from the Sobolev-Korn inequality of Lemma~\ref{eq: sobolev-korn 1}.

For the induction step, we fix $\ell\geq 0$, ${\alphabold}=(\alpha_0,\dots,\alpha_\Qrm)$ with $\alpha_i>0$, $\Omega$ and $R_{\Qrm+1,\ell}$, and we assume that both conclusions of the proposition hold for $\widetilde{\alphabold}=(\alpha_0,\dots,\alpha_{\Qrm-1})$  and  any $\ell'\in\NN$.  

We first show that the estimate for $\norm{\Mbold_1\grad\otimes \ubold}_\Mvarcal$ holds when $\ubold\in\BDVM{}{}[\Omega,\Mbold_1,\Tcal^{\ell}(\RR^d)]$ (otherwise the estimate is obvious since $\TDVM{\Qrm+1,\ell}{\alphabold}[\ubold,\Mcal]=+\infty$). Using the map $R_{\Qrm,\ell+1}$, \cref{lem 3.9 brehol}, the continuous embeddings
\[
\BDVM{}{}[\Omega,\Mbold_1,\Tcal^{\ell+1}(\RR^d)] 
\hookrightarrow
\LL{d/(d-1)}[\Omega,\Tcal^{\ell+1}(\RR^d)]
\hookrightarrow
\LL{1}[\Omega,\Tcal^{\ell+1}(\RR^d)]
\]
and the induction hypotheses, we get for $\wbold\in\BDVM{}{}[\Omega,\Mbold_1,\Tcal^{\ell+1}(\RR^d)]$ the following estimates:
\[
\begin{aligned}
\norm{\Mbold_1\grad\otimes \ubold}_\Mvarcal
&\leq
C_1 \left( \norm{\Mbold_1\grad\otimes \ubold - R_{\Qrm,\ell+1} \wbold}_\Mvarcal + \norm{\ubold}_1\right)
\\
&\leq
C_2 \left( \norm{\Mbold_1\grad\otimes \ubold - \wbold}_\Mvarcal +\norm{\wbold- R_{\Qrm,\ell+1} \wbold}_{d/(d-1)} + \norm{\ubold}_1\right)
\\
&\leq
C_3 \left( \norm{\Mbold_1\grad\otimes \ubold - \wbold}_\Mvarcal +\TDVM{\Qrm,\ell+1}{\widetilde{\alphabold}}[\wbold,\Mcal] + \norm{\ubold}_1\right)
\\
&\leq
C_4 \left( \alpha_\Qrm \norm{\Mbold_1\grad\otimes \ubold - \wbold}_\Mvarcal +\TDVM{\Qrm,\ell+1}{\widetilde{\alphabold}}[\wbold,\Mcal] + \norm{\ubold}_1\right)
\end{aligned}
\] 
for suitable $C_1,\,C_2,\,C_3,\,C_4>0$.
By taking the minimum over all $\wbold\in\BDVM{}{}[\Omega,\Mbold_1,\Tcal^\ell(\RR^d)]$  we get
\[
\norm{\Mbold_1\grad\otimes \ubold}_\Mvarcal \leq C_4\left( \norm{\ubold}_1 + \TDVM{\Qrm+1,\ell}{\alphabold}[\ubold,\Mcal]\right),
\]
via the minimum representation in \cref{rem: minimal representation remark}.

For the coercivity estimate, assume that it is not true, i.e.\ there exists $(\ubold_j)_{j\in\NN}$ such that each $\ubold_j\in\LL{d/(d-1)}[\Omega,\Tcal^\ell(\RR^d)]$ and
\[
\norm{\ubold_j - R_{\Qrm+1,\ell} \ubold_j}_{d/(d-1)} =1
\quad\text{and}\quad
\TDVM{\Qrm+1,\ell}{\alphabold}[\ubold_j,\Mcal]\leq j^{-1}.
\]
Since $\ker(\TDVM{\Qrm+1,\ell}{\alphabold}[\ubold_j,\Mcal]) =\Im(R_{\Qrm+1,\ell})$ then for each $j$ it holds
\[
\TDVM{\Qrm+1,\ell}{\alphabold}[\ubold_j- R_{\Qrm+1,\ell} \ubold_j,\Mcal] = \TDVM{\Qrm+1,\ell}{\alphabold}[\ubold_j,\Mcal].
\]
Also, since the first estimate holds, then
\[
\norm{\Mbold_1 \grad\otimes (\ubold_j-R_{\Qrm+1,\ell} \ubold_j)}_\Mvarcal 
\leq 
C_4 
\left(
\TDVM{\Qrm+1,\ell}{\alphabold}[\ubold_j,\Mcal] + \norm{\ubold_j - R_{\Qrm+1,\ell} \ubold_j}_1
\right)
\]
and $(\ubold_j - R_{\Qrm+1,\ell} \ubold_j)_{j\in\NN}$ is bounded in $\BDVM{}{}[\Omega,\Mcal,\Tcal^\ell(\RR^d)]$ by the continuous embedding.
By the compact embedding there exists a subsequence of $(\ubold_j - R_{\Qrm+1,\ell} \ubold_j)_{j\in\NN}$, not relabelled, converging to $\ubold^\ast\in\LL{1}[\Omega,\Tcal^\ell(\RR^d)]$ with $R_{\Qrm+1,\ell} \ubold^\ast = 0$ since
$R_{\Qrm+1,\ell} (\ubold_j - R_{\Qrm+1,\ell} \ubold_j)=0$  for all $j$.
Moreover, the lower semi-continuity leads to
\[
0 \leq \TDVM{\Qrm+1,\ell}{\alphabold}[\ubold^\ast,\Mcal] 
\leq 
\liminf_{j\to\infty} \TDVM{\Qrm+1,\ell}{\alphabold}[\ubold_j,\Mcal] = 0.
\]
This means that $\ubold^\ast\in\ker(\TDVM{\Qrm+1,\ell}{\alphabold})$ and
$
\Mbold_1\grad\otimes (\ubold_j - R_{\Qrm+1,\ell} \ubold_j) \to \bm{0}
$
in $\Mvarcal(\Omega,\Tcal^{\ell+1}(\RR^d))$ 
with
$(\ubold_j - R_{\Qrm+1,\ell} \ubold_j)\to \bm{0}$ in $\BDVM{}{}[\Omega,\Mbold_1,\Tcal^{\ell}(\RR^d)]$ and in $\LL{d/(d-1)}[\Omega,\Tcal^\ell(\RR^d)]$ by the continuous embedding.
This contradicts $\norm{\ubold_j - R_{\Qrm+1,\ell} \ubold_j}_{d/(d-1)}=1$ for all $j$ and the coercivity holds.
\end{proof}

The next proposition, similar to \cite[Proposition 4.1]{BreHol2014}, proves the coercivity of the minimisation problem \cref{eq: minimization problem}.

\begin{proposition}\label{prop: corecivity 4.1 brehol} 
Let $p\in [1,\infty[$ with $p \leq d/(d-1)$ and $F:\LL{p}[\Omega,\Tcal^\ell(\RR^d)]\to]-\infty,\infty]$. If $F$ is bounded from below and there exist an onto projection $R$ as in \cref{def: R} such that for each sequence $(\ubold_j)_j\in{\NN}$ with $\ubold_j\in\LL{d/(d-1)}[\Omega,\Tcal^\ell(\RR^d)]$ it holds
\[
\norm{R \ubold_j}_{d/(d-1)} \to \infty
\quad
\text{and}
\quad
\left(\norm{\ubold_j - R\ubold_j}_{d/(d-1)}\right)_{j\in\NN} 
\text{ is bounded}
\quad
\Rightarrow
\quad 
F(\ubold_j) \to \infty,
\]
then $\TDVM{\Qrm,\ell}{\alphabold} + F$ is coercive in $\LL{p}[\Omega,\Tcal^\ell(\RR^d)]$.
\end{proposition}
\begin{proof}
Let $(\ubold_j)_{j\in\NN}$ be a sequence such that each $\ubold_j\in\LL{p}[\Omega,\Tcal^\ell(\RR^d)]$ and if $(F(\ubold_j)+\TDVM{\Qrm,\ell}{\alphabold}[\ubold_j,\Mcal])_{j\in\NN}$ is bounded then $(\ubold_j)_{j\in\NN}$ is bounded. 
Since $F$ is bounded from below by assumption, then the sequences $(F(\ubold_j))_{j\in\NN}$ and $(\TDVM{\Qrm,\ell}{\alphabold}[\ubold_j,\Mcal])_{j\in\NN}$ are bounded too.
Thus, the boundedness of $(\TDVM{\Qrm,\ell}{\alphabold}[\ubold_j,\Mcal])_{j\in\NN}$ implies that each $\ubold_j\in\LL{d/(d-1)}[\Omega,\Tcal^\ell(\RR^d)]$ by the continuous embedding in \cref{th: bredies 4.16}.
Now, let $R$ be a projection map as in \cref{def: R} such that the hypotheses holds. Thus, there exists a constant $C>0$ such that:
\[
\norm{\ubold_j-R \ubold_j}_{d/(d-1)}
\leq 
C\TDVM{\Qrm,\ell}{\alphabold}[\ubold_j,\Mcal],\quad\text{for all  }j\in\NN
\]
and the sequence $(\norm{\ubold_j- R \ubold_j}_{d/(d-1)})_{j\in\NN}$ is bounded. Note that $(\norm{R\ubold_j}_{d/(d-1)})_{j\in\NN}$ is bounded too otherwise $(F(\ubold_j))_{j\in\NN}$ results unbounded and contradicts the hypothesis.
From the continuous embedding of Lebesgue spaces, then $(\ubold_j)_{j\in\NN}$ is bounded in $\LL{p}[\Omega,\Tcal^\ell(\RR^d)]$.
\end{proof}

We are now ready to prove the following existence theorem, similarly to \cite[Theorem 4.2]{Bredies2012}:
\begin{theorem}\label{th: existence 4.2 brehol}
Let $p\in[1,\infty[$ with $p\leq d/(d-1)$ and assume that $F:\LL{p}[\Omega,\Tcal^\ell(\RR^d)]\to]-\infty,\infty]$ is proper, convex, lower semi-continuous and coercive as in \cref{prop: corecivity 4.1 brehol}.
Then there exists a solution to the problem
\begin{equation}
\min_{
\ubold\in\LL{p}[\Omega,\Tcal^\ell(\RR^d)]
} 
\TDVM{\Qrm,\ell}{\alphabold}[\ubold,\Mcal] + F(\ubold).
\label{eq: existence 4.2 brehol}
\end{equation}
Furthermore, if $\ubold\in\BDVM{}{}[\Omega,\Tcal^\ell(\RR^d)]$ is such that $F(\ubold)<\infty$ then the minimum is finite.
\end{theorem}
\begin{proof}
We note immediately that the regulariser $\TDVM{\Qrm,\ell}{\alphabold}[\ubold,\Mcal]$ is finite if and only if $\ubold\in\BDVM{}{}[\Omega,\Mcal,\Tcal^\ell(\RR^d)]$, otherwise it is trivial to prove that a minimiser exists and the minimum is equal to $+\infty$.
Thus, assume $F(\ubold)<\infty$ for some $\ubold\in\BDVM{}{}[\Omega,\Mcal,\Tcal^\ell(\RR^d)]$ and consider a minimising sequence $(\ubold_j)_{j\in\NN}$ for $G = F + \TDVM{\Qrm,\ell}{\alphabold}$. Note that such sequence exists since $G$ is bounded from below.
Now, applying the coercivity result in \cref{prop: corecivity 4.1 brehol} for a $p'\in[p,d/(d-1)]$ and $p'>1$, then there exists a subsequence of $(\ubold_j)$, weakly convergent to  $\ubold^\ast\in\LL{p}[\Omega,\Tcal^\ell(\RR^d)]$.
Moreover, since $G$ is convex and lower semi-continuous, we get that $\ubold^\ast$ is a minimiser by weak lower semi-continuity and by assuming that $G$ is proper, the minimum is finite.
\end{proof}

From \cref{th: existence 4.2 brehol}, we can conclude as in \cite[Corollary 4.3]{BreHol2014} that there exists a solution for the minimisation problem \cref{eq: existence 4.2 brehol} in the context of inverse problems, i.e.\ when the fidelity term $F(\ubold)$ is defined from a forward operator $\Scal : \LL{p}[\Omega,\Tcal^\ell(\RR^d)]\to Y$, linear and continuous in a normed space $Y$, and the observed data $\ubold^\diamond\in Y$ as:
\[
F(\ubold) = \frac{1}{q}\norm{\Scal \ubold - \ubold^\diamond}_Y^q,\quad\text{for }q\in[1,\infty[.
\]
Of course, for a strictly convex norm $\norm{\blank}_Y$ the uniqueness of the solution depends on the injectivity of $\Scal$: in general, uniqueness does not hold since $\TDVM{\Qrm,\ell}{\alphabold}$ is not strictly convex. 
\section{Conclusions}
In this work, we have introduced and analysed the \emph{total directional variation} of arbitrary order, providing a precise framework to extend the notions of \emph{total generalized variation} \cite{BreKunPoc2010} and \emph{directional total variation} \cite{directionaltv}. In particular, we have proven a representation formula for the total directional variation of arbitrary order, which is a key for the design of a primal-dual algorithm which can be used in many imaging applications, see \cite{ParMasSch18applied}. 
\section*{Acknowledgements}
The authors are grateful to 
Prof.\ Jan Lellman, University of L\"{u}beck, (Germany) and Dr.\ Martin Holler, University of Graz (Austria) for the useful discussions. 

\bibliographystyle{siamplain}
\bibliography{biblio}

\end{document}